\documentclass[12pt]{amsart}
\usepackage{amssymb}
\usepackage{comment}
\newtheorem{theorem}{Theorem}[section]
\newtheorem{lemma}[theorem]{Lemma}
\newtheorem{proposition}[theorem]{Proposition}
\newtheorem{corollary}[theorem]{Corollary}

\theoremstyle{definition}

\newtheorem{examp}[theorem]{Example}

\newtheorem{remar}[theorem]{Remark}

\numberwithin{equation}{section}

\newenvironment{remark}{\begin{remar}\rm}{\gzaun\end{remar}}
\newcommand{\gzaun}{\unskip\nobreak\hfil\penalty50%
\hskip1em\hbox{}\nobreak\hfil%
$\#$\parfillskip=0pt\finalhyphendemerits=0}

\newenvironment{example}{\begin{examp}\rm}{\diams\end{examp}}
\newcommand{\diams}{\unskip\nobreak\hfil\penalty50%
\hskip1em\hbox{}\nobreak\hfil%
$\diamondsuit$\parfillskip=0pt\finalhyphendemerits=0}

\newcommand{\bfind}[1]{\index{#1}{\bf #1}}
\newcommand{\n}{\par\noindent}

\newcommand{\sn}{\par\smallskip\noindent}
\newcommand{\mn}{\par\medskip\noindent}
\newcommand{\bn}{\par\bigskip\noindent}
\newcommand{\pars}{\par\smallskip}
\newcommand{\parm}{\par\medskip}
\newcommand{\parb}{\par\bigskip}
\newcommand{\isom}{\simeq}

\newcommand{\Aut}{\mbox{\rm Aut}\,}

\newcommand{\card}{\mbox{\rm card}\,}

\newcommand{\trdeg}{\mbox{\rm trdeg}\,}
\newcommand{\supp}{\mbox{\rm supp}}
\newcommand{\appr}{\mbox{\rm appr}}

\newcommand{\rr}{\mbox{\rm rr}\,}

\newcommand{\cO}{\mathcal O}
\newcommand{\cA}{\mathcal A}
\newcommand{\cC}{\mathcal C}

\newcommand{\cT}{\mathcal T}

\newcommand{\cN}{\mathcal N}

\newcommand{\cV}{\mathcal V}
\newcommand{\cL}{\mathcal L}
\newcommand{\cS}{\mathcal S}

\newcommand{\N}{\mathbb N}
\newcommand{\Q}{\mathbb Q}
\newcommand{\R}{\mathbb R}
\newcommand{\Z}{\mathbb Z}

\newcommand{\bd}{{\bf d}}
\newcommand{\bh}{{\bf h}}
\newcommand{\bhsc}{\mbox{\scriptsize\bf h}}
\begin{document}

\title[Approximation types]{Approximation types describing extensions of valuations 
to rational function fields}

\author{F.-V.~Kuhlmann}
\address{Institute of Mathematics, University of Szczecin, ul.~Wielkopolska 15
70-451 Szczecin, Poland}
\email{fvk@usz.edu.pl}
\thanks{The author would like to thank Enric Nart, Giulio Peruginelli and Anna Rzepka 
for reading the manuscript and providing comments and corrections.}


\date{November 19, 2021}

\begin{abstract}
We introduce the notion of {\it approximation type} for the partial, and in certain cases
the total description of extensions of a given valuation from a field $K$ to the rational
function field $K(x)$. To every extension, a unique approximation type of $x$ over $K$ is 
associated, while $x$ may be the limit of many pseudo Cauchy sequences. Approximation types
also provide information in cases where the extensions are not immediate, and we prove that 
they correspond bijectively to the extensions when $K$ is algebraically closed or lies dense 
in its algebraic closure.
\end{abstract}

\maketitle

\tableofcontents

%
%
\section{Introduction}
In this paper we will work with (Krull) valuations on fields and their extensions to
rational function fields. As we will show that under certain natural conditions, these 
extensions are uniquely determined by what we call {\it approximation types}, it is
important to note from the start that we {\it always identify equivalent valuations}.
For basic information on valued fields and for notation, see Section~\ref{sectprel}.

Take a valued field $(K,v_0)$. It is an important task to describe, analyze and 
classify all extensions $v$ of the valuation $v_0$ from $K$ to the rational function
field $K(x)$. In order to be able to compute the value of every element of $K(x)$ 
with respect to $v$, it suffices to be able to compute
the value of all polynomials in $x$, that is, we only have to deal with the polynomial
ring $K[x]$. Indeed, if $f,g\in K[x]$, then necessarily, $v\frac f g =vf-vg$. We know the
values of all elements in $K$. If in addition we know the 
value $vx$, then everything would be easy if for every polynomial
\begin{equation}                                   \label{pol}
f(x)\>=\>\sum_{i=0}^n c_ix^i\in K[x]
\end{equation}
the following equation would hold:
\begin{equation}                                            \label{minval}
vf(x)\>=\> \min_{0\leq i\leq n} v_0 c_i + ivx\>.
\end{equation}
Indeed, we can define valuations on $K(x)$ in this way by choosing $vx$ to be any 
element in some ordered abelian group which contains $vK$. If we choose $vx=0$, we 
obtain the \bfind{Gau{\ss} valuation}.

But what if Equation~\ref{minval} does not always hold? Then there are 
{\it polynomials of unexpected value}, the value of which is larger than the minimum 
of the values of its monomials. This observation has led to the theory of {\it key polynomials}, on which by now an abundant number of articles are available. 

\parm
Let us give a basic classification of all extensions $v$ of the valuation $v_0$ of 
$K$ to $K(x)$. The \bfind{rational rank} of an ordered abelian 
group $\Gamma$ is $\rr \Gamma:= \dim_\Q \Q\otimes_{\Z} \Gamma\>$ (note that
$\Q\otimes_{\Z} \Gamma$ is the \bfind{divisible hull} of $\Gamma$). As the 
{\it Abhyankar Inequality}
\begin{equation}                            \label{111}
1\>=\>\trdeg K(x)|K\>\geq\>\rr vK(x)/v_0 K\,+\,\trdeg K(x)v|Kv_0
\end{equation}
holds by Proposition~\ref{prelBour}, there are the following mutually
exclusive cases:
\sn
$\bullet$ \ $(K(x)|K,v)$ is \bfind{valuation-algebraic} (case $1>0+0$): \par
$vK(x)/v_0K$ is a torsion group and $K(x)v|Kv_0$ is algebraic,
\sn
$\bullet$ \ $(K(x)|K,v)$ is \bfind{value-transcendental} (case $1=1+0$): \par
$vK(x)/v_0K$ has rational rank 1, but $K(x)v|Kv_0$ is algebraic,
\sn
$\bullet$ \ $(K(x)|K,v)$ is \bfind{residue-transcendental} (case $1=0+1$): \par
$K(x)v|Kv_0$ has transcendence degree 1, but $vK(x)/v_0 K$ is a torsion
group.
\mn
We will combine the value-transcendental case and the
residue-transcendental case by saying that
\sn
$\bullet$ \ $(K(x)|K,v)$ is \bfind{valuation-transcendental}: \par
$vK(x)/v_0K$ has rational rank 1, or $K(x)v|Kv_0$ has transcendence
degree 1.
\sn
A special case of the valuation-algebraic case is the following:
\sn
$\bullet$ \ $(K(x)|K,v)$ is {\bf immediate}: \par
$vK(x)=v_0K$ and $K(x)v=Kv_0$.
\sn
For more details on this notion, see Section~\ref{sectie}.

\begin{remark}                              \label{limitrem}
It was observed by several authors that an immediate extension
of $v_0$ from $K$ to $K(x)$ can be represented as a limit of an infinite
sequence of residue-transcendental extensions; see e.g.\ \cite{[APZ3]}.
The approach is particularly important because residue-transcendental
extensions behave better than other types of extensions: the
corresponding extensions of value group and residue field are finitely
generated (see Corollary~2.7 in \cite{KTrans}), and they do not generate a defect
that is not already present in $(K,v_0)$:
see the Generalized Stability Theorem in \cite{[K8]} and its applications
in \cite{[K--K1],[K7]}.
\end{remark}

\parb
We denote by $\tilde K$ the algebraic closure of $K$. We will assume throughout that 
$v$ is extended to $\widetilde{K(x)}$; this also induces an extension of $v$ from $K$ 
to $\tilde K$. In this way, we associate to $(K(x)|K,v)$ the extension $(\tilde{K}(x)
|\tilde{K},v)$. Note that the value group of an algebraically closed valued field is
divisible, and its residue field is also algebraically closed. Further, $v\tilde{K}/
v_0K$ is a torsion group, and $\tilde{K}v|Kv_0$ is algebraic.

If $K$ is algebraically closed, then $v_0K$ is divisible and $Kv_0$ is
algebraically closed. In this case, for any extension $(K(x)|K,v)$ there
are only the following mutually exclusive cases:
\sn
$(K(x)|K,v)$ is immediate: \ $vK(x)=v_0K$ and $K(x)v=Kv_0$,\n
$(K(x)|K,v)$ is value-transcendental: \ $\rr vK(x)/v_0K=1$,
but $K(x)v=Kv_0$,\n
$(K(x)|K,v)$ is residue-transcendental: \ $\trdeg K(x)v|Kv_0=1$,
but $vK(x)=v_0K$.
\sn
As a consequence, if $(K,v_0)$ is an arbitrary valued field and we have an arbitrary
extension $v$ of $v_0$ to $K(x)$, then 
\sn
$\bullet$ \ $(K(x)|K,v)$ is valuation-algebraic if and only if\n
\qquad $(\tilde{K}(x)|\tilde{K},v)$ is immediate,
\n
$\bullet$ \ $(K(x)|K,v)$ is value-transcendental if and only if\n
\qquad $(\tilde{K}(x)|\tilde{K},v)$ is value-transcendental,
\n
$\bullet$ \ $(K(x)|K,v)$ is residue-transcendental if and only if\n
\qquad $(\tilde{K}(x)|\tilde{K},v)$ is residue-transcendental.

\parm
Ostrowski in \cite{Os} and Kaplansky in \cite{[Ka]} gave us a powerful tool for the
analysis and the construction of immediate extensions: {\it pseudo Cauchy sequences}
(also called {\it pseudo convergent sequences}). As is the case for Cauchy sequences,
an element in a valued field extension that is a limit of a pseudo Cauchy sequence 
will in general be a limit of many different pseudo Cauchy
sequences. It is therefore desirable to have a unique object that can readily be 
assigned to an element in a valued field extension and that contains all information 
that is contained in pseudo Cauchy sequences, and possibly more. We will define such 
objects, called {\it approximation types}, in Section~\ref{sectat}. They are nests of
ultrametric balls; for basic definitions and properties of ultrametric balls, see
Section~\ref{sectballs}. In a given extension $(K(x)|K,v)$, the 
approximation type of $x$ over $K$ is obtained by intersecting all ultrametric balls
with center $x$ in $(K(x),v)$ with $K$.

In Section~\ref{sectiat} we define {\it immediate approximation types}, and in
Section~\ref{sectiatpCs} we describe how to obtain them from pseudo Cauchy sequences, 
and vice versa. In Section~\ref{sectalgtriat} we prove the analogues of Kaplansky's basic
theorems for pseudo Cauchy sequences in the language of approximation types. Immediate 
approximation types were introduced and extensively used in \cite{KV}.

Pseudo Cauchy sequences and immediate approximation types are suitable tools to describe 
immediate extensions $(K(x)|K,v)$, but they do not carry enough information to describe 
extensions that are not immediate. For that reason, the new notions of {\it pseudo 
monotone sequences} and {\it pseudo divergent sequences} were introduced and used in 
\cite{Ch,PS2}. However, like pseudo Cauchy sequences they are not unique objects 
associated with the valuation $v$ and the element $x$. Moreover, they appear to be 
somewhat unnatural constructs for capturing the necessary information. Therefore, we 
propose to use the more uniform approximation types in place of all of these 
sequences. Still it should be noted that like these sequences, the approximation types 
cannot describe all extensions when $K$ is not algebraically closed. Nevertheless, we will
define in Sections~\ref{sectpure} and~\ref{sectalpure} important subclasses of extensions 
which can be fully described.

\pars
We fix a valuation $v_0$ on $K$ and denote by the letter\footnote{Note that $\cV$ is 
the calligraphic version of the capital letter V, and certainly {\it not} a capital
greek letter nu, as the capital version of $\nu$ in the Greek alphabet is N. 
Addressing $\cV$ as ``capital nu'' would be something capitally new.} 
$\cV$ the set of all extensions of
$v_0$ to $K(x)$. To every $v\in\cV$ we then associate the approximation type of $x$ 
with respect to $v$, which we denote by $\appr_v (x,K)$. By $\cA$ we denote the 
set of all non-trivial (abstract) approximation types over $(K,v_0)$. We will prove:
\begin{theorem}                       \label{MTat1}
Let $K(x)$ be the rational function field in one variable over $K$. Then for every 
non-trivial approximation type {\bf A} over $(K,v_0)$ there is an extension $v$ of $v_0$ 
to $K(x)$ such that $\appr_v (x,K)={\bf A}$. In other words, the function
\begin{equation}                            \label{V->A}
\cV\>\longrightarrow \cA\>, \qquad  v\,\mapsto\, \appr_v (x,K) 
\end{equation} 
is surjective.
\end{theorem}

\begin{theorem}                       \label{MTat2}
Let $K(x)$ be the rational function field in one variable over $K$.
Assume that $K$ is algebraically closed or that $(K,v_0)$ lies dense in its algebraic
closure. Then the valuations $v\in\cV$ are fully characterized by the approximation 
types $\appr_v (x,K)$. In particular, the function (\ref{V->A}) is a bijection.
\end{theorem}

These theorems will be proved in Sections~\ref{sectmth1} and~\ref{sectmth2}, where we will 
also make precise what we mean by ``fully characterized''. The classification of extensions 
we introduced above is clearly reflected in the approximation types, together with 
related information necessary to fully describe the valuations.

Theorem~\ref{MTat2} is a special case of a more general theorem. Using the notion of
{\it pure extension} which we introduced in \cite{KTrans}, we show in 
Theorem~\ref{MTatenh2} that the statement of Theorem~\ref{MTat2} holds for all pure 
extensions, provided that $vK$ is divisible. The surjectivity of the function 
remains valid when we restrict to a natural subset of the set of all approximation types.
We obtain Theorem~\ref{MTat2} for the 
case of algebraically closed $K$ from Proposition~\ref{acpure} which states that if 
$K$ is algebraically closed, then every extension $v\in\cV$ is pure.

In order to cover the case where $(K,v_0)$ lies dense in its algebraic closure for 
some and hence for all extensions of $v_0$ to $\tilde{K}$ (see Lemma~\ref{densext}), 
we generalize the notion {\it pure extension} to what we call {\it almost pure
extension}. In Proposition~\ref{denspure} we will then show that in this case every
extension $v\in\cV$ is almost pure. As we will show that Theorem~\ref{MTatenh2} also
holds for almost pure extensions, Theorem~\ref{MTat2} will follow from it. 

\pars
Two important cases where $(K,v_0)$ lies dense in its algebraic closure are:
\sn
1) The rank of $(K,v_0)$ is $1$, i.e., $v_0 K$ is archimedean ordered (hence order
embeddable in $\R$), and its henselization is algebraically closed. The latter occurs 
when $v_0 K$ is divisible and $Kv_0$ is algebraically closed of characteristic $0$
(see \cite[32.21]{W2}).
\sn
2) The field $K$ is separable-algebraically closed and $v_0$ is non-trivial (see
\cite[Theorem~1.11]{Ku6}).

\parm
An alternative approach to the proof of Theorem~\ref{MTat1} can be given by use of model
theory. Under the minimal necessary additional assumptions we show in Theorem~\ref{mtreal}
that every non-trivial approximation type {\bf A} over $(K,v_0)$ can be realized in some 
elementary extension $(K^*,v)$ of $(K,v_0)$. For the proof we show that under the 
additional assumptions, the approximation types are subsets of suitable $1$-types over
$(K,v_0)$. This in fact explains the choice of the name ``approximation type''.

\pars
Note that for simplicity we will often write ``$v$'' in place of ``$v_0$'' for the 
valuation on $K$ when it is not necessary to distinguish it from its extensions.

\bn
%
%
%
\section{Preliminaries}                        \label{sectprel}
By a valued field $(K,v)$ we mean a field $K$ endowed with a Krull valuation $v$. This is
a function from $K$ to $\Gamma\cup\{\infty\}$ where $\Gamma$ is an ordered abelian group 
and $\infty$ an element larger than all elements of $\Gamma$, which satisfies $v(a)=
\infty\Leftrightarrow a=0$, $v(ab)=v(a)+v(b)$ and $v(a+b)\geq \min\{v(a),v(b)\}$ for 
all $a,b\in K$. From these laws one deduces that $v(1)=0$, $v(-a)=v(a)$, $v(a^{-1})=-v(a)$ 
and $v(a_1+\ldots
+a_n)\geq\min_{1\leq i\leq n} v(a_i)$ with equality holding if all values $v(a_i)$ are
distinct. We will make use of all these facts freely.

The value group $v(K^\times)\subseteq\Gamma$ of $(K,v)$ will be denoted by $vK$, and its
residue field 
$\{a\in K\mid v(a)\geq 0\}/\{a\in K\mid v(a)> 0\}$ by $Kv$. The value of an element $a$ 
will be denoted by $va$ in place of $v(a)$, and its residue by $av$. By $(L|K,v)$ we
denote a field extension $L|K$ where $v$ is a valuation on $L$ and $K$ is endowed with the
restriction of $v$. For background on valuation theory, see \cite{[En],[EP],[ZS]}.

\begin{lemma}                                 \label{approx}
Take any extension $(K(x)|K,v)$. For each $c\in K$, the following assertions are equivalent:
\sn
a) there is some $c'\in K$ such that $v(x-c')>v(x-c)$;
\sn
b) there is $d\in K$ such that $vd(x-c)=0$ and $d(x-c)v\in Kv$.
\sn
If $d(x-c)v\in Kv$ holds for some $d\in K$ such that $vd(x-c)=0$, then it holds for 
all such $d$.
\end{lemma}
\begin{proof}
Assume first that there is some $c'\in K$ such that $v(x-c')>v(x-c)$. Then 
$v(x-c)=\min\{v(x-c'),v(x-c)\}=v(c-c')\in vK$, so that there is $d\in K$ such that 
$vd(x-c)=0$. It follows 
that $v(d(x-c)-d(c'-c))=vd+v(x-c')>vd(x-c)$, whence $d(x-c)v=d(c'-c)v\in Kv$.

\pars
Now assume that there is $d\in K$ such that $vd(x-c)=0$ and $d(x-c)v\in Kv$. Pick 
$b\in K$ such that $vb=0$ and $bv=d(x-c)v$. Then $v(d(x-c)-b)>0$, whence 
$v(x-c-bd^{-1})>-vd=v(x-c)$, so that $c':=c+bd^{-1}$ satisfies assertion a).

\pars
Assume that there is $d\in K$ such that $vd(x-c)=0$ and $d(x-c)v\in Kv$ and take 
some $d'\in K$ such that $vd'(x-c)=0$. Then $d'(x-c)v=(d'd^{-1}v)(d(x-c)v)\in Kv$
since $vd'd^{-1}=0$.
\end{proof}

\mn
%
%
\subsection{Ordered sets and cuts}                           
\mbox{ }\sn
Take any totally ordered set $(T,<)$, an element $\alpha\in T$ and two subsets $D,E\subseteq
T$. We write 
\n
$\alpha>D$ if $\alpha>\delta$ for all $\delta\in D$, 
\n
$\alpha<E$ if $\alpha<\varepsilon$ for all $\varepsilon\in E$, 
\n
$D<E$ if $\delta<\varepsilon$ for all $\delta\in D$ and $\varepsilon\in E$,
\n
and similarly for ``$\leq$'' in place of ``$<$''. A \bfind{cut} {\bf in} $T$ is a pair 
$(D,E)$ of subsets of $T$ such that $D<E$ and $D\cup E=T$. If this is the case, then $D$ 
is an \bfind{initial segment} of $T$, that is, if $\delta \in D$ and $\alpha\leq\delta$, 
then $\alpha\in D$, and $E$ is a \bfind{final segment} of $T$, that is, if $\varepsilon 
\in E$ and $\alpha\geq\varepsilon$, then $\alpha\in E$.

If $\alpha\notin T$ is an element in some totally ordered set containing $(T,<)$, then we 
say that \bfind{$\alpha$ realizes the cut $(D,E)$} if $D<\alpha<E$. The \bfind{cut induced
by $\alpha$ in $T$} is the pair $(\{\delta\in T\mid\delta<\alpha\}\,,\,\{\varepsilon\in T
\mid\varepsilon>\alpha\})$.

\mn
%
%
\subsection{Immediate extensions}                           \label{sectie}
\mbox{ }\sn
An extension $(L|K,v)$ is called \bfind{immediate} if 
the canonical embeddings $vK\hookrightarrow vL$ and
$Kv\hookrightarrow Lv$ are both onto. Instead, we will
also say ``if $(K,v)$ and $(L,v)$ have the same value group and the same
residue field'' or just ``if $vL=vK$ and $Lv=Kv$'' (recall that we are
identifying equivalent valuations and places, so we may view $vK$ as a
subgroup of $vL$ and $Kv$ as a subfield of $Lv$). However, the reader should
note that this is less precise and can be misunderstood. For instance,
if $vK\isom\Z$ and $L|K$ is finite, then still, $vK\isom\Z$ even if the
embedding of $vK$ in $vL$ is not onto. Important examples of immediate extensions
are henselizations and completions of valued fields. 

\begin{lemma}                               \label{ievagus}
An extension $(L|K,v)$ is immediate if and only if for every $x\in L\setminus K$ 
there is $a\in K$ such that $v(x-a)>vx$.
\end{lemma}
\begin{proof}
%
``$\Rightarrow$'': Assume that $(L|K,v)$ is immediate and take $x\in L\setminus K$. Then
$x\ne 0$ and therefore, $vx\in vL=vK$. Thus there is $d\in K$ such that $vdx=0$ and 
$dxv\in Lv=Kv$. Now we apply Lemma~\ref{approx} with $c=0$
to obtain $c'\in K$ such that $v(x-c')>vx$. Thus $a:=c'$ is the required element.
\sn
``$\Leftarrow$'': Take $\alpha\in vL$ and $x\in L$ such
that $vx=\alpha$. If there is $a\in K$ such that $v(x-a)>vx$, then
$\alpha=vx=va\in vK$. Now let $\zeta\in Lv$ and $x\in L$ be such that
$xv=\zeta$. If there is $a\in K$ such that $v(x-a)>vx=0$, then
$\zeta=xv=av\in Kv$.
\end{proof}

The assertion of this lemma is an important observation, as it allows us to generalize 
the definition of ``immediate extension'' to more general valued structures for which 
the invariants associated to them can be much more complicated than the pair of value
group and residue field. One of such cases occurs when we consider a valued field 
extension $(L|K,v)$ and a $K$-vector space $V\subseteq L$. When equipped with the 
restriction of the valuation $v$ of $L$, we call $(V,v)$ a \bfind{valued 
$(K,v)$-vector space}. If $V\subseteq V'\subseteq L$, then we call $(V'|V,v)$ an
\bfind{immediate extension of valued vector spaces} if for every $a\in V'$ there is 
$b\in V$ such that $v(a-b)>va$.

\mn
%
%
\subsection{Algebraic valuation independence}                           
\mbox{ }\sn
The following proposition has turned out to be amazingly universal in many different
applications of valuation theory. It plays an important role for example in
algebraic geometry as well as in the model theory of valued fields, in
real algebraic geometry, or in the structure theory of exponential Hardy
fields (= nonarchimedean ordered fields which encode the asymptotic
behaviour of real-valued functions including $\exp$ and $\log$).
For more details and the easy but technical proof of the proposition, see
\cite[Chapter VI, \S10.3, Theorem~1]{Bou}.
\begin{proposition}                                      \label{prelBour}
Let $(L|K,v)$ be an extension of valued fields. Take elements $x_i,y_j
\in L$, $i\in I$, $j\in J$, such that the values $vx_i\,$, $i\in I$,
are rationally independent over $vK$, and the residues $y_jv$, $j\in
J$, are algebraically independent over $Kv$. Then the elements
$x_i,y_j$, $i\in I$, $j\in J$, are algebraically independent over $K$.

Moreover, if we write
\begin{equation}                            \label{polBour}
f\>=\> \displaystyle\sum_{k}^{} c_{k}\,
\prod_{i\in I}^{} x_i^{\mu_{k,i}} \prod_{j\in J}^{} y_j^{\nu_{k,j}}\in
K[x_i,y_j\mid i\in I,j\in J]
\end{equation}
in such a way that whenever $k\ne\ell$, then
there is some $i$ s.t.\ $\mu_{k,i}\ne\mu_{\ell,i}$ or some $j$ s.t.\
$\nu_{k,j}\ne\nu_{\ell,j}\,$, then
\begin{equation}                            \label{value}
vf\>=\>\min_k\, v\,c_k \prod_{i\in I}^{}
x_i^{\mu_{k,i}}\prod_{j\in J}^{} y_j^{\nu_{k,j}}\>=\>
\min_k\, vc_k\,+\,\sum_{i\in I}^{} \mu_{k,i} v x_i\>.
\end{equation}
That is, the value of the polynomial $f$ is equal to the least of the
values of its monomials. In particular, this implies:
\begin{eqnarray*}
vK(x_i,y_j\mid i\in I,j\in J) & = & vK\oplus\bigoplus_{i\in I}
\Z vx_i\\
K(x_i,y_j\mid i\in I,j\in J)v & = & Kv\,(y_jv\mid j\in J)\>.
\end{eqnarray*}
Moreover, the valuation $v$ and the residue map on $K(x_i,y_j\mid i\in I,j\in J)$ are
uniquely determined by their restriction to $K$, the values $vx_i$ and
the residues $y_jv$.
\parm
Conversely, if $(K,v)$ is any valued field, the elements $x_i,y_j$, $i\in I$, $j\in J$, 
are algebraically independent over $K$, and we assign to the $vx_i$
any values in an ordered group extension of $vK$ which are rationally
independent, then (\ref{value}) defines a valuation on 
$K(x_i,y_j\mid i\in I,j\in J)$, and the
residues $y_jv$, $j\in J$, are algebraically independent over $Kv$. \qed
\end{proposition}

The proof of the following corollary is straightforward:
\begin{corollary}                          \label{pBcor}
Take a valued field $(K,v_0)$ and an element $z$ transcendental over $K$. Take an element 
$\alpha$ either in $v_0K$, or in some ordered abelian group containing $v_0K$. In the latter 
case, assume that $\alpha$ is not a torsion element modulo $v_0K$. For every polynomial
$f(X)=\sum_{i=0}^n c_iX^i\in K[X]$, define 
\[
vf(z) \>:=\> \min_i v_0 c_i + i\alpha\>,
\]
and extend $v$ to $K(z)$ in the canonical way. Then $v$ is a valuation on $K(z)$. 

If $\alpha\in v_0 K$ and $d\in K$ is such that $v_0 d=-vz$, then $dzv$ is 
transcendental over $Kv_0\,$, $K(z)v=Kv_0 (dzv)$ and $vK(z)=v_0 K$. If 
$\alpha\notin v_0 K$, then $vK(z)=v_0K\oplus\Z\alpha$ and $K(z)v=Kv_0\,$.
\end{corollary}

We need to know when the construction described in this corollary still yields the same 
valuation even when $z$ and $\alpha$ are changed. Although we have already announced that 
we identify equivalent valuations, at this point we will be a bit more precise. If $v_1$ 
and $v_2$ are extensions of a valuation $v_0$ from $K$ to some extension field $L$, then 
we will say that $v_1$ and $v_2$ are \bfind{equivalent over $v_0$} if there is an order
preserving isomorphism $\rho$ from $v_1 L$ to $v_2 L$ which is the identity on $v_0 K$ 
such that $v_2=\rho\circ v_1\,$.
\begin{lemma}                           \label{cutord}
Assume that $\Gamma$ is an ordered abelian group and $\alpha\notin\Gamma$ is an
element in some ordered abelian group containing $\Gamma$. Then the ordering on $\Gamma
\oplus\Z\alpha$ is uniquely determined by the cut that $\alpha$ induces in $\Gamma$, 
provided that \n
a) $\Gamma$ is divisible, or 
\n
b) $\alpha>\Gamma$ (in which case the induced cut is $(\Gamma,\emptyset)$), or
\n
c) $\alpha<\Gamma$ (in which case the induced cut is $(\emptyset,\Gamma)$).
\end{lemma}
\begin{proof}
We have to show that it can be deduced from the cut $(D,E)$ induced by $\alpha$ in $\Gamma$
whether any given element $\gamma+n\alpha\in \Gamma\oplus\Z\alpha$ with $\gamma\in\Gamma$ 
and $n\in\Z\setminus\{0\}$ is positive or not. 

Assume first that $\Gamma$ is divisible, so 
that $\frac{\gamma}{n}\in\Gamma$. If $n$ is positive, then $\gamma+n\alpha>0\Leftrightarrow 
\alpha > -\frac{\gamma}{n}\Leftrightarrow -\frac{\gamma}{n}\in D$. If $n$ is negative, then 
$\gamma+n\alpha>0\Leftrightarrow \alpha <-\frac{\gamma}{n}\Leftrightarrow -\frac{\gamma}{n}
\in E$. 

If $\alpha>\Gamma$, then $\gamma+n\alpha>0\Leftrightarrow n>0$. If $\alpha<\Gamma$, then 
$\gamma+n\alpha>0\Leftrightarrow n<0$.
\end{proof}

\begin{lemma}                           \label{exteq}
Take a valued field $(K,v_0)$ and an element $z$ transcendental over $K$. Pick any 
$\alpha_1$ and $\alpha_2$ in some ordered abelian groups containing $v_0K$.
For $i=1,2$, assume that $\alpha_i$ is not a torsion element modulo $v_0 K$ if $\alpha_i
\notin v_0 K$ and extend $v_0$ to $K(z)$ by using Corollary~\ref{pBcor}, assigning the 
value $\alpha_i$ to $z$. Then the following assertions hold:
\sn
1) The valuations $v_1$ and $v_2$ can only be equivalent if either both $\alpha_1\notin 
v_0K$ and $\alpha_2\notin v_0K$ or both $\alpha_1\in v_0K$ and $\alpha_2\in v_0K$.
\sn
2) Assume that $\alpha_1\notin v_0K$ and $\alpha_2\notin v_0K$. If $v_1$ and $v_2$ are 
equivalent over $v_0$, then $\alpha_1$ and $\alpha_2$ realize the same cut $(D,E)$ 
in $v_0K$. The converse holds if $v_0 K$ is divisible, or $D=\emptyset$, or $E=\emptyset$. 
\end{lemma}
\begin{proof}
1): If one of $\alpha_1,\alpha_2$ lies in $v_0K$ and the other does not, then by 
Proposition~\ref{prelBour}, one of $v_1K(z)$, $v_2K(z)$ is equal to $v_0K$ and the other is 
not, so $v_1$ and $v_2$ cannot be equivalent.

\sn
2): Assume that $\alpha_1\notin v_0K$ and $\alpha_2\notin v_0K$. Further, assume first 
that $v_1$ and $v_2$ are equivalent over $v_0$, that is, there is an order preserving
isomorphism $\rho$ from $v_1K(z)$ to $v_2K(z)$ fixing $v_0K$ such that $v_2=\rho\circ 
v_1$. Thus $\alpha_2=v_2 z=\rho\alpha_1\in v_2K(z)$.  Since $\rho$ is order preserving
and fixes $v_0 K$, 
this shows that $\alpha_1$ and $\alpha_2$ realize the same cut in $v_0 K$.

For the converse, assume that $\alpha_1$ and $\alpha_2$ realize the same cut $(D,E)$ in 
$v_0 K$. First, assume that $v_0 K$ is divisible. This implies that $\alpha_1\notin v_0K$ 
and $\alpha_2\notin v_0K$ are not torsion elements modulo 
$v_0 K$. Thus sending $\alpha_1$ to $\alpha_2$ induces an isomorphism $\rho$ from 
$v_0 K\oplus\Z\alpha_1$ to $v_0 K\oplus\Z\alpha_2$ which leaves $v_0 K$ fixed. By 
Lemma~\ref{cutord}, there is a unique ordering on $v_0 K\oplus\Z\alpha_1$ determined by the 
cut $(D,E)$. Through $\rho$ it induces on $v_0 K\oplus\Z\alpha_2$ an ordering, which again 
by Lemma~\ref{cutord} must coincide with the ordering determined by $(D,E)$. This shows 
that $\rho$ is order preserving. By Corollary~\ref{pBcor}, the extension of $v_0$ to $K(z)$ 
is uniquely determined by the value $\alpha_2=\rho\alpha_1$ assigned to the element 
$z$, so $\rho \circ v_1=v_2\,$, showing that $v_1$ and $v_2$ are equivalent over $v_0\,$.

Now assume that $D=\emptyset$ (the case of $E=\emptyset$ is treated analogously). Then again,
$\alpha_1\notin v_0K$ and $\alpha_2\notin v_0K$ are not torsion elements modulo $v_0 K$, and the proof proceeds as before.
\end{proof}

\begin{remark}
More generally, the converse in part 2) of the lemma always holds when the archimedean class 
of any element realizing the cut is not equal to the archimedean class of any element in 
$\Gamma$. This happens if and only if there is a convex subgroup $\Delta$ of $\Gamma$ which 
is cofinal in $D$ or coinitial in $E$.
\end{remark}

\mn
%
%
%
\subsection{The sets $v(x-K)$}                           \label{sectva-K}
\mbox{ }\sn
Take an extension $(L|K,v)$ and an element $x\in L$. In this section we will 
investigate the set
\[
v(x-K)\>:=\> \{v(x-c)\mid c\in K\}\>\subseteq\> vL\cup\{\infty\}\>.
\]

The following facts were proved in \cite[Lemma 3.1]{B}:
\begin{lemma}                     \label{maxapp}
Take a valued field extension $(L|K,v)$ and $x\in L$. 
\sn
1) The set $v(x-K)\cap vK$ is an initial segment of $vK$.
\sn
2) The set $v(x-K)\setminus vK$ has at most one element.
\sn
3) If $\alpha\in v(x-K)\setminus vK$, then $\alpha=\max v(x-K)$ and 
\[
v(x-K)\cap vK=\{\gamma\in vK\mid \gamma<\alpha\}\>,
\]
which is the lower cut set of the cut induced by $\alpha$ in $vK$.
\sn
4) For every $c\in K$, $\{\gamma\in v(x-K)\mid \gamma<v(x-c)\}$ is a subset of $vK$ and
thus an initial segment of $vK$.
\sn
5) We have $v(x-c)=\max v(x-K)$ if and only if $v(x-c)\notin vK$ or $(d(x-c))v\notin 
Kv$ for every $d\in K$ such that $v(d(x-c))=0$.
\end{lemma}
\begin{proof}
Take any $c\in K$ and $\gamma\in vK$ such that $\gamma<v(x-c)$. Pick $c_\gamma\in K$
such that $vc_\gamma=\gamma$. Then $c+c_\gamma\in K$ and $v(x-(c+c_\gamma))= 
\min\{v(x-c),vc_\gamma\}=vc_\gamma=\gamma$, so $\gamma\in v(x-K)\cap vK$. This proves 
part 1) and the inclusion ``$\supseteq$'' in the second assertion of part 3).
\sn
Further, take $c_1,c_2\in K$ such that $v(x-c_1)<v(x-c_2)$. Then 
\[
v(x-c_1)\>=\>\min\{v(x-c_1),v(x-c_2)\}\>=\>v(c_2-c_1)\,\in\, vK\>,
\]
so the assertion of part 2) as well as the first assertion of part 3) and the 
inclusion ``$\subseteq$'' in its second assertion must hold.
\sn
4): If $v(x-c)\in vK$, then the assertion follows from part 1), and otherwise from part 3). 
\sn
5): This is the contrapositive of Lemma~\ref{approx}.
\end{proof}

\begin{lemma}                                     \label{va-K}
Take a valued field extension $(L|K,v)$ and elements $x,y\in L$. 
\sn
1) If $v(x-K)$  has no maximal element, then it is an initial segment of $vK$.
\sn
2) If $(K(x)|K,v)$ is immediate, then $v(x-K)$ has no maximal element.
\sn
3) If for all $x\in L$, $v(x-K)$ has no maximal element, then the extension 
$(L|K,v)$ is immediate.
\sn
4) If $v(x-y)>v(x-K)$, then $v(x-c)=v(y-c)$ for all $c\in K$, and $v(x-K)=v(y-K)$.
\sn
5) If $v(x-K)$  has no maximal element, then the following are equivalent:
\n
a) \ $v(x-y)>v(x-K)$,
\n
b) \ $v(x-y)\geq v(x-K)$,
\n
c) \ $v(x-c)=v(y-c)$ for all $c\in K$.
\sn
6) Take extensions $(L|K,v)$ and $(L(x)|L,v)$, and assume that $vL=vK$. If 
$v(x-K)\ne v(x-L)$, then for some $y\in L$, $v(x-y)>v(x-K)$ and $v(x-K)=v(y-K)$.
\end{lemma}
\begin{proof}
1): If it has no maximal element, then $v(x-K)$ is the union of the sets 
$\{\gamma\in v(x-K)\mid \gamma<v(x-c)\}$ where $c$ runs through all elements of $K$.
By part 4) of Lemma~\ref{maxapp} all of these sets are initial segments of $vK$, hence 
so is their union.
\sn
2): This follows from part 5) of Lemma~\ref{maxapp}.
\sn
3): This follows from Lemma~\ref{ievagus}.
\sn
4): If $c\in K$, then by assumption, $v(x-y)>v(x-c)$, which implies that 
$v(y-c)=\min\{v(x-y),v(x-c)\}=v(x-c)$. As this holds for all $c\in K$, we also 
obtain that $v(x-K)=v(y-K)$.

\sn
5): Assume that $v(x-K)$  has no maximal element. Then assertions a) and b) are
trivially equivalent. The implication a)$\Rightarrow$c) is part 4) of our lemma. 
Now suppose that assertion b) does not hold, and pick some 
$c\in K$ such that $v(x-c)>v(x-y)$. It follows that 
\[
v(y-c)\>=\>\min\{v(x-y),v(x-c)\}\>=\>v(x-y)\>\ne\> v(x-c)\>,
\]
so assertion c) does not hold.

\sn
6): Since $K\subseteq L$, we have that $v(x-K)\subseteq v(x-L)$. 
Assume that $v(x-K)\ne v(x-L)$. Then there exists $y\in L$ such that 
$v(x-y)\notin v(x-K)$. Suppose that there is some 
$c\in K$ such that $v(x-c)>v(x-y)$. Then $v(x-y)=\min\{v(x-y),v(x-c)\}=v(y-c)\in 
vL=vK$; but then $v(x-y)\in v(x-K)$ by part 4) of Lemma~\ref{maxapp}, a contradiction.
This proves that $v(x-y)>v(x-K)$. Hence by part 4) of the present lemma, 
$v(x-K)=v(y-K)$.
\end{proof}

\pars
Note that the converse of part 2) of this lemma does in general not hold.

\mn
%
%
\subsection{Pseudo Cauchy sequences}              \label{sectpcs}
\mbox{ }\sn
Take a valued field $(K,v)$ and a sequence
$(c_\nu)_{\nu<\lambda}$ of elements in $K$, indexed by ordinals
$\nu<\lambda$ where $\lambda$ is a limit ordinal. It is called a
\bfind{pseudo Cauchy sequence} (or \bfind{pseudo convergent sequence}) if
\sn
{\bf (PCS)} \ $v(c_{\tau}-c_{\sigma})\> >\>v(c_{\sigma}-c_\rho)\>$
whenever $\rho<\sigma<\tau<\lambda\>$.
\sn
We will say
that an assertion \bfind{holds ultimately} for $(c_\nu)_{\nu<\lambda}$
if there is $\nu_0< \lambda$ such that the assertion holds for all
$c_\nu$ with $\nu\geq\nu_0\,$.

We set\glossary{$\gamma_\nu$}
\[
\gamma_\nu\>:=\>v(c_{\nu+1}-c_\nu)\>.
\]
If $(c_\nu)_{\nu<\lambda}$ is a pseudo Cauchy sequence, then
$(\gamma_\nu)_{\nu<\lambda}$ is strictly increasing.

\begin{lemma}                               \label{proppCsf}
Let $(c_\nu)_{\nu<\lambda}$ be a pseudo Cauchy sequence in
$(K,v)$. Then
\begin{equation}                            \label{gammanuf}
v(c_\nu-c_\mu)\,=\,\gamma_\mu\;\mbox{ whenever }\mu<\nu<\lambda\>.
\end{equation}
If $x\in K$, then either\n
\begin{equation}                            \label{ylimitf}
v(x-c_\mu)\,<\,v(x-c_\nu)\;\mbox{ whenever }\mu<\nu<\lambda\>,
\end{equation}
or there is $\nu_0<\lambda$ such that
\[
v(x-c_\nu)\,=\,v(x-c_{\nu_0})\mbox{ whenever } \nu_0\leq\nu<\lambda\>.
\]
Property (\ref{ylimitf}) is equivalent to
\begin{equation}                            \label{ylimit=f}
v(x-c_\nu)\,=\,\gamma_\nu\;\mbox{ for all }\nu<\lambda\>.
\end{equation}
In other words, if $(v(x-c_\nu))_{\nu<\lambda}$ is not strictly
increasing, then it is ultimately constant.
\end{lemma}

Taking $x=0$, we obtain:
\begin{corollary}                           \label{incfixval}
For every pseudo Cauchy sequence $(c_\nu)_{\nu<\lambda}$, either
$(vc_\nu)_{\nu<\lambda}$ is strictly increasing or ultimately constant.
\end{corollary}

Note that if $(L|K,v)$ is an extension of valued fields and
$(c_\nu)_{\nu<\lambda}$ is a pseudo Cauchy sequence in $(K,v)$, then it
is also a pseudo Cauchy sequence in $(L,v)$. An element $x\in L$ is
called a {\bf pseudo limit} (or just {\bf limit}) of
$(c_\nu)_{\nu<\lambda}$\index{limit of a pseudo Cauchy sequence} if it
satisfies (\ref{ylimitf}), or equivalently, (\ref{ylimit=f}). Since
$v(x-c_{\nu+1})\geq\gamma_{\nu+1}> \gamma_\nu$ implies that
$v(x-c_\nu)=\min\{\gamma_\nu,v(x-c_{\nu+1})\}= \gamma_\nu\,$, both
conditions are equivalent to
\begin{equation}                 \label{limitpCs}
v(x-c_\nu)\,\geq\,\gamma_\nu\;\mbox{ for all }\nu<\lambda\>.
\end{equation}

We will only be interested in pseudo Cauchy sequences in $(K,v)$ that have no limit
in $K$ and therefore do not have a last element. This justifies that from the start 
we have indexed pseudo Cauchy sequences by limit ordinals.

The following is Theorem 1 of \cite{[Ka]}.
\begin{theorem}                                   \label{KT1}
If $(L|K,v)$ is an immediate extension, then every $a\in L\setminus K$ is limit of 
a pseudo Cauchy sequence in $(K,v)$ without a limit in $K$.
\end{theorem}
\n
An analogue of this theorem for 
immediate approximation types will be proved in Lemma~\ref{imme}.

\pars
We consider a pseudo Cauchy sequence $(c_\nu)_{\nu<\lambda}$ 
in $(K,v)$ and a polynomial $f\in K[X]$. We will say that 
$(c_\nu)_{\nu<\lambda}$ {fixes the value of} $f$\index{pseudo Cauchy sequence
fixes value of a polynomial} if the sequence $vf(c_\nu)_{\nu<\lambda}$
is ultimately constant. If $(c_\nu)_{\nu<\lambda}$ fixes the value
of every polynomial in $K[X]$, then it is said to be of
{\bf transcendental type}.\index{pseudo Cauchy sequence of transcendental type} 
If there is some $f\in K[X]$ whose value is not fixed by $(c_\nu)_{\nu<\lambda}$, 
then $(c_\nu)_{\nu<\lambda}$ is said to be of {\bf algebraic
type}.\index{pseudo Cauchy sequence of algebraic type}
The following are Theorems 2 and 3 of \cite{[Ka]}.
\begin{theorem}                                   \label{KT2}
For every pseudo Cauchy sequence $(c_\nu)_{\nu<\lambda}$ in
$(K,v)$ of transcendental type there exists a simple immediate transcendental
extension $(K(x),v)$ such that $x$ is a limit of $(c_\nu)_{\nu<\lambda}$.
If $(K(y),v)$ is another extension field of $(K,v)$ such that $y$
is a limit of $(c_\nu)_{\nu<\lambda}$, then $y$ is also transcendental over $K$ 
and the isomorphism between $K(x)$ and $K(y)$ over $K$ sending $x$ to $y$ is
valuation preserving.
\end{theorem}

\begin{theorem}                                       \label{KT3}
Take a pseudo Cauchy sequence $(c_\nu)_{\nu<\lambda}$ in $(K,v)$ of algebraic type 
and a polynomial $f(X)\in K[X]$ of minimal degree whose value is not fixed by 
$(c_\nu)_{\nu<\lambda}$. If $\,a\,$ is a root of $\,f$, then there
exists an extension of $v$ from $K$ to $K(a)$ such that $(K(a)|K,v)$ is
an immediate extension and $a$ is a limit of $(c_\nu)_{\nu<\lambda}$.

If $(K(b),v)$ is another extension field of $(K,v)$ such that
$b$ is a limit of $(c_\nu)_{\nu<\lambda}$, then any field isomorphism between 
$K(a)$ and $K(b)$ over $K$ sending $a$ to $b$ will preserve the valuation.
\end{theorem}

In Section~\ref{sectalgtriat} we will prove analogues of the last two theorems for 
immediate approximation types. These will also provide proofs of the above theorems 
through the connection we set up between pseudo Cauchy sequences and approximation
types in Section~\ref{sectiatpCs}.

\mn
%
%
\subsection{Ultrametric balls and nests}              \label{sectballs}
\mbox{ }\sn
We define the \bfind{closed ultrametric ball} in $(K,v)$ of radius $\gamma\in
vK\infty:=vK \cup\{\infty\}$ centered at $c\in K$ to be 
\[
B_\gamma (c,K) \>=\> \{a\in K \mid v(a-c)\geq\gamma\}\>,
\]
and the \bfind{open ultrametric ball} in $(K,v)$ of radius $\gamma\in
vK$ centered at $c\in K$ to be 
\[
B_\gamma^\circ (c,K) \>=\> \{a\in K \mid v(a-c)>\gamma\}\>.
\]
Note that under the topology induced by the valuation, both types of balls are 
open and closed. Note further that all of these balls contain their center and are
thus nonempty.

\begin{lemma}                                  \label{ballbasic}
1) If $B=B_\gamma (c,K)$ or $B=B_\gamma^\circ (c,K)$, and if $b\in B$, then 
$B=B_\gamma (b,K)$ or $B=B_\gamma^\circ (b,K)$, respectively. In other words, every element
in an (open or closed) ultrametric ball is its center.
\sn
2) Any two closed or open ultrametric balls $B,B'$ are either disjoint or 
comparable by inclusion.
\end{lemma}
\begin{proof}
1): If $b\in B_\gamma (c,K)$, then $v(c-b)\geq\gamma$, hence for every $a\in
B_\gamma (c,K)$, we have that $v(a-b)\geq\min\{v(a-c),v(b-c)\}\geq\gamma$. This 
proves that $B_\gamma (c,K)\subseteq B_\gamma (b,K)$. The reverse inclusion 
follows by symmetry.

The proof for the balls $B_\gamma^\circ (c,K)$ is analogous.
\sn
2): Assume that $B\cap B'\ne\emptyset$, and choose $c\in B\cap B'$. Then by part 1),
$c$ is a center of both $B$ and $B'$. If one of the two has a smaller radius than 
the other, then by definition it contains the other. If both have the same radius,
then the closed one contains the open one. The case of $B=B'$ is trivial.
\end{proof}

A \bfind{nest of balls} is a nonempty collection of closed and open balls linearly
ordered by inclusion. A \bfind{full nest of balls} is a nest of balls $\cN$ that
contains every closed or open ball which contains some ball $B\in\cN$, i.e.,
\[
\cN\>=\>\{B'\mid B'\mbox{ open or closed ultrametric ball containing some } B\in\cN\}\>.
\]
Part 2) of the above lemma shows that the set on the right hand side is indeed a 
nest. For any nest $\cN$ of balls, we set
\begin{equation}                         \label{bigcapN}
\bigcap\cN\>:=\>\bigcap_{B\in\cN} B\>.
\end{equation}
If a full nest $\cN$ contains a smallest ball $B$, then $\bigcap\cN=B$ and $\cN$ is 
\bfind{generated by} $B$ in the sense that $\cN$ consists of exactly all open and closed 
balls that contain $B$.

\begin{lemma}                                \label{uniqnest}
Take a nest of balls $\cN$.
\sn
1) If $B\in\cN$ is a ball with center $c$ and radius $\gamma$, then every other 
ultrametric ball in $\cN$ that contains $B$ is a closed or open ultrametric ball 
of radius $\leq\gamma$ with center $c$; that is, every larger ball that appears in $\cN$
is uniquely determined by $B$.
\sn
2) For every $\gamma \in vK\infty$, the nest $\cN$ contains at most one closed ball
and at most one open ball of radius $\gamma$.
\sn
3) There is a uniquely determined full nest $\cN'$ containing $\cN$ and such that 
each of its ultrametric balls contains an ultrametric ball from $\cN$. The nest
$\cN'$ satisfies $\bigcap\cN'=\bigcap\cN$.
\end{lemma}
\begin{proof}
1): This follows from part 1) of Lemma~\ref{ballbasic}.
\sn
2): Since every two balls in a nest are comparable, this follows from part 1) of our
lemma.
\sn
3): The collection of all ultrametric balls that contain some ultrametric ball 
from $\cN$ is a nest by part 2) of Lemma~\ref{ballbasic}. It is clear that it is a 
full nest of balls. The last assertion follows from the fact that $\cN\subseteq\cN'$
and every ball in $\cN'$ contains a ball from $\cN$.
\end{proof}

Pseudo Cauchy sequences give rise to nests of balls:
\begin{lemma}                                  \label{pCs->nest}
Take a pseudo Cauchy sequence $(c_\nu)_{\nu<\lambda}$ in $(K,v)$. Then 
\begin{equation}                                 \label{pCsnest}
(B_{\gamma_\nu}(c_\nu,K))_{\nu<\lambda}
\end{equation} 
is a nest of balls in $K$. The intersection over all balls in this nest is the set of
all limits of $(c_\nu)_{\nu<\lambda}$ in $K$.
\end{lemma}
\begin{proof}
Assume that $\mu<\nu<\lambda$. Then by (\ref{gammanuf}), 
$v(c_\nu-c_\mu)=\gamma_\mu\,$, that is, $c_\nu\in B_{\gamma_\mu}(c_\mu,K)$. By part 2)
of Lemma~\ref{ballbasic}, this implies that (\ref{pCsnest}) is a nest of balls.

An element $c\in K$ is a limit of $(c_\nu)_{\nu<\lambda}$ if and only if for all 
$\nu<\lambda$ we have that $v(c-c_\nu)=\gamma_\nu\,$, or in other words,
$c\in B_{\gamma_\nu}(c_\nu,K)$. This in turn holds if and only if $c$ lies in the
intersection over all balls in the nest. This proves our second assertion.
\end{proof}

\bn
%
%
\section{Approximation types}  \label{sectat}    
%
%
%
\subsection{Definition of approximation types}  \label{sectatdef}
\mbox{ }\sn
An \bfind{approximation type} {\bf A over} $(K,v)$ is either a full nest of open and
closed balls in $(K,v)$, or the empty set. It follows that 
\[
\supp{\bf A}\>:=\> \{\gamma \mid \mbox{{\bf A} contains a closed ball of radius }
\gamma\}
\]
is a (possibly empty) initial segment of $vK\infty$, called the \bfind{support} of
{\bf A}. If $\gamma\in\supp{\bf A}$, then by part 2) of Lemma~\ref{uniqnest}, {\bf A}
contains a unique closed ball of radius $\gamma$, which we will denote by 
${\bf A}_\gamma\,$. If {\bf A} also contains an open ball of radius $\gamma$, then
that too is unique, and we will denote it by ${\bf A}_\gamma^\circ\,$. 

We may write \glossary{${\bf A}_\gamma$} ${\bf A}_\gamma=B_\gamma (c_\gamma, K)$ if 
$\gamma\in \supp{\bf A}$, and ${\bf A}_\gamma=\emptyset$ otherwise. Likewise, we may write 
\glossary{${\bf A}_\gamma^\circ$} ${\bf A}_\gamma^\circ=B_\gamma^\circ (c_\gamma,K)\,$ if 
{\bf A} contains an open ball of radius $\gamma$, and ${\bf A}_\gamma^\circ=\emptyset$
otherwise; note that if {\bf A} contains an open ball of radius $\gamma$, then by
Lemma~\ref{ballbasic} we can choose $c_\gamma$ to be any of its elements, and since the 
open ball of radius $\gamma$ is contained in the closed ball ${\bf A}_\gamma$, again by
Lemma~\ref{ballbasic} we can take the same $c_\gamma$ also as a center of  
${\bf A}_\gamma$. If $\gamma$ is not the maximal 
element of $\supp{\bf A}$, then ${\bf A}_\gamma^\circ\ne\emptyset$, but if 
$\gamma$ is the maximal element of $\supp{\bf A}$, then ${\bf A}_\gamma^\circ$ may or 
may not be empty. Further, note that for every $\gamma\in vK$ and $c\in K$, the balls 
$B_\gamma (c,K)$ and $B_\gamma^\circ (c,K)$ are nonempty, but they may not be contained in
{\bf A}; in particular, $B_\gamma (c,K)\in {\bf A}$ {\it does not} imply 
$B_\gamma^\circ (c,K)\in {\bf A}$. See Example~\ref{exreve} below.

\pars
We note the following straightforward observation:
\begin{lemma}                             \label{ateq}
Two approximation types {\bf A} and ${\bf A}'$ over $(K,v)$ are equal if and only if
$\supp{\bf A}=\supp{\bf A}'$ and for all $\gamma\in\supp{\bf A}$, ${\bf A}_\gamma=
{\bf A}'_\gamma$ and ${\bf A}_\gamma^\circ={{\bf A}'_\gamma}^\circ$.
\end{lemma}

\pars
Recall that $\bigcap {\bf A}$ denotes the intersection over all balls in {\bf A}.
If ${\bf A}=\emptyset$, we set $\supp{\bf A}=\emptyset$ and $\bigcap {\bf A}=K$. In this 
case, conditions like $\alpha>\supp {\bf A}$ and $\alpha\geq\supp {\bf A}$ shall be 
understood to always be satisfied. 

When we say that $\supp{\bf A}$ has no maximal element, then we will tacitly mean that 
it is nonempty. If this is the case, then for every $\gamma\in\supp{\bf A}$
there is a larger $\beta\in\supp{\bf A}$, so ${\bf A}_\beta\ne\emptyset$.
Now part 1) of Lemma~\ref{uniqnest} proves:
\begin{lemma}                                           \label{suppnle}
If ${\bf A}\ne\emptyset$ and $\supp{\bf A}$ has no maximal element, then {\bf A} 
is uniquely determined by 
its closed balls, and it is also uniquely determined by its open balls. In this case,
\[
\bigcap {\bf A} \>=\> \bigcap_{\displaystyle\gamma\in\supp {\bf A}} {\bf A}_\gamma
\>=\> \bigcap_{\displaystyle\gamma\in \supp {\bf A}} {\bf A}_\gamma^\circ\>.
\]

\end{lemma}

\parm
Take any extension $(L|K,v)$ and $x\in L$. For all
$\gamma\in vK\infty$, we set
\begin{equation}                            \label{atalpha}
\appr_v (x,K)_{\gamma} \>:=\> \{c\in K\mid v(x-c)\geq\gamma\}
\>=\>B_\gamma(x,L)\cap K\>,
\end{equation}
and for $\gamma\in vK$, we set
\begin{equation}                            \label{atalphacirc}
\appr_v(x,K)_{\gamma}^\circ\>:=\> \{c\in K\mid v(x-c)>\gamma\}
\>=\>B_\gamma^\circ(x,L)\cap K\>.
\end{equation}
Further, we set
\begin{eqnarray*}
\appr_v(x,K) &:=& \{\appr_v(x,K)_{\gamma}\mid \gamma\in vK\infty
\mbox{ and } \appr_v(x,K)_{\gamma}\ne\emptyset\}\\
&& \cup \,\{\appr_v(x,K)_{\gamma}^\circ\mid \gamma\in vK
\mbox{ and } \appr_v(x,K)_{\gamma}^\circ\ne\emptyset\}\>.
\end{eqnarray*}
Note that $\appr_v(x,K)=\emptyset$ if and only $vx<vK$.

\begin{remark}
As the right hand sides of (\ref{atalpha}) and (\ref{atalphacirc}) show, subtraction is 
not needed to define the approximation type of an element. Therefore, these approximation
types can already be defined in ultrametric spaces without any further algebraic 
structure and can be used to study extensions of ultrametric spaces and other structures 
with underlying ultrametric spaces. 
\end{remark}


%
\begin{lemma}
For each $\gamma\in vK\infty$, $B_\gamma(x,L)\cap K$ is a closed 
ultrametric ball and $B_\gamma^\circ(x,L)\cap K$ is an open ultrametric ball in
$(K,v)$, if nonempty. The collection $\appr_v (x,K)$ is an approximation type over 
$(K,v)$. 
\end{lemma}
\begin{proof}
Assume that $c\in B_\gamma(x,L)\cap K$. Then $v(x-c)\geq\gamma$ and by the 
ultrametric triangle law, for every $d\in K$ we have
\[
d\in B_\gamma(x,L)\>\Leftrightarrow\>v(x-d)\geq\gamma \>\Leftrightarrow\>
v(c-d)\geq\gamma \>\Leftrightarrow\> d\in B_\gamma(c,K)\>,
\]
whence $B_\gamma(x,L)\cap K=B_\gamma(c,K)$.
A similar argument yields that if $c\in B_\gamma^\circ(x,L)\cap K$, then
$B_\gamma^\circ(x,L)\cap K=B_\gamma^\circ(c,K)$.

Our assertion that $\appr_v(x,K)$ is an approximation type follows from the facts that 
if $B_\gamma^\circ(x,L)\cap K\ne\emptyset$, then $B_\gamma(x,L)\cap K\ne\emptyset$,
and if $\gamma>\beta\in vK$ and $B_\gamma(x,L)\cap K\ne\emptyset$, then
$B_\beta(x,L)\cap K\ne\emptyset$.
\end{proof}
\n
We call $\appr_v(x,K)$ the \bfind{approximation type of $x$ over $(K,v)$}.
For the sake of completeness, we state the following criteria for the equality of 
approximation types.
\begin{lemma}                                
Take an extension $(L|K,v)$ and $x,x'\in L$. Then
\sn
$\appr_v(x,K)=\appr_v(x',K) \,\Rightarrow\, v(x-x')\geq
\supp\,\appr_v(x,K)=\supp\,\appr_v(x',K)\>.$
\sn
Conversely, if 
$v(x-x')>\supp\,\appr_v(x,K)\cup\supp\,\appr_v(x',K)$ or if 
$v(x-x')\in vK$ and $v(x-x')>\supp\,\appr_v(x,K)$, then 
$\appr_v(x,K)=\appr_v(x',K)$.
\end{lemma}
\begin{proof}
Assume that $\appr_v(x,K)=\appr_v(x',K)$. Then 
\[
\supp\,\appr_v(x,K)\>=\>\supp\,\appr_v(x',K)
\]
and for every $\gamma\in \supp\,\appr_v(x,K)$ we have that  
$\appr_v(x,K)_\gamma=\appr_v(x',K)_\gamma$. The latter implies that for 
$c\in\appr_v(x,K)_\gamma$, $v(x-c)\geq\gamma$ and $v(x'-c)\geq\gamma$, whence
$v(x-x')\geq\gamma$. 

\pars
Now assume that $v(x-x')>\supp\,\appr_v(x,K)\cup\supp\,\appr_v(x',K)$. Take 
$\gamma\in \supp\,\appr_v(x,K)$. For every $c\in\appr_v(x,K)_\gamma$ we have
that $v(x-c)\geq\gamma$ and since $v(x-x')>\gamma$, we find that $v(x'-c)\geq\gamma$.
This shows that $\appr_v(x,K)_\gamma\subseteq\appr_v(x',K)_\gamma$ and $\gamma
\in\supp\,\appr_v(x',K)$. In particular, 
\begin{equation}                    \label{ssubs}
\supp\,\appr_v(x,K)\>\subseteq\>\supp\,\appr_v(x',K)\>.
\end{equation}
By symmetry, we can interchange $x$ and $x'$, so we find that 
$\supp\,\appr_v(x,K)=\supp\,\appr_v(x',K)$ and $\appr_v(x,K)_\gamma=
\appr_v(x',K)_\gamma$ for every $\gamma\in\supp\,\appr_v(x,K)$.
For each such $\gamma$ and every $c\in\appr_v(x,K)_\gamma^\circ$ we have
that $v(x-c)>\gamma$, and since $v(x-x')>\gamma$, it follows that $v(x'-c)>\gamma$.
This shows that $\appr_v(x,K)_\gamma^\circ\subseteq\appr_v(x',K)_\gamma^\circ$, and 
by symmetry, we obtain equality. From Lemma~\ref{ateq} we now obtain that $\appr_v(x,K)=
\appr_v(x',K)$.

\pars
Finally, assume that $v(x-x')\in vK$ and $v(x-x')>\supp\,\appr_v(x,K)$. As above, we show 
that (\ref{ssubs}) holds. Suppose that the reverse inclusion were not true. Then there is 
some $\gamma\in \supp\,\appr_v(x',K)$ such that $\gamma>\supp\,\appr_v(x,K)$. We pick some 
$c\in K$ such that $v(x'-c)=\gamma$. Now $v(x-c)\geq\min\{v(x-x'),v(x'-c)\}=:\delta$. By 
our assumption, $\delta\in vK$ with $\delta>\supp\,\appr_v(x,K)$. However, $v(x-c)\geq
\delta$ implies that $c\in\appr_v(x,K)_\delta\,$, a contradiction. Hence the supports of 
the two approximation types are equal.

For every $\gamma\in \supp\,\appr_v(x,K)$ we have that $v(x-x')>\gamma$ and therefore, 
$v(x-c)\geq\gamma\Leftrightarrow v(x'-c)\geq\gamma$ and $v(x-c)>\gamma\Leftrightarrow 
v(x'-c)>\gamma$, showing that $\appr_v(x,K)_\gamma=\appr_v(x',K)_\gamma$ and 
$\appr_v(x,K)_\gamma^\circ=\appr_v(x',K)_\gamma^\circ$. Again from Lemma~\ref{ateq} we  
obtain that $\appr_v(x,K)=\appr_v(x',K)$.
\end{proof}

The next lemma demonstrates the connection between the supports of approximation types 
and the sets $v(x-K)$.
\begin{lemma}                                \label{suppvx-K}
Take an extension $(L|K,v)$ and $x\in L$. Then
\begin{equation}
\supp\,\appr_v(x,K)\>=\>v(x-K)\cap vK\>.
\end{equation}
\end{lemma}
\begin{proof}
Take $\gamma\in\supp\,\appr_v(x,K)$, hence $\appr_v(x,K)_\gamma\ne\emptyset$. Pick 
some element $c\in \appr_v(x,K)_\gamma\,$. Then $v(x-c)\geq\gamma$. Now part 1) of 
Lemma~\ref{maxapp} shows that $\gamma\in v(x-K)\cap vK$. 

For the converse, take $\gamma\in v(x-K)\cap vK$ and choose $c\in K$ such that $\gamma
=v(x-c)$. Then $c\in\appr_v(x,K)_\gamma$ and therefore, $\gamma\in\supp\,\appr_v(x,K)$.
\end{proof}

\pars
If {\bf A} is an approximation type over $(K,v)$ and there exists an
element $x$ in some valued extension field $(L,v)$ such that ${\bf A} =
\appr_v(x,K)$, then we say that \bfind{$x$ realizes {\bf A}} (in $(L,v)$).
If {\bf A} is realized by some $c\in K$, then {\bf A} will be called
{\bf trivial}.\index{trivial approximation type} This holds if and only
if ${\bf A}_\infty\ne \emptyset$ (i.e., $\infty\in\supp {\bf A}$), 
in which case ${\bf A}_\infty=\{c\}$.
As ${\bf A}_\infty$ can contain at most one element, a trivial
approximation type over $(K,v)$ can be realized by only one element in $K$.

\mn
%
%
%
\subsection{Immediate approximation types}     \label{sectiat}
\mbox{ }\sn
Take an approximation type {\bf A} over $(K,v)$. Then {\bf A} will be called
{\bf immediate} \index{immediate approximation type} if it is nonempty and
\[
\bigcap {\bf A} \>=\>\emptyset\>.
\]
If {\bf A} is immediate, then ${\bf A}\ne\emptyset$ and $\supp {\bf A}$ cannot have 
a maximal element, hence 
Lemma~\ref{suppnle} shows that {\bf A} is uniquely determined by its nonempty closed 
ultrametric balls. In particular, an immediate approximation type cannot be trivial.

\pars
To simplify notation, we can represent immediate approximation types as
\[
{\bf A}\>=\> \{{\bf A}_\gamma\mid\gamma\in\supp {\bf A}\}\>,
\]
and if ${\bf A}=\appr_v(x,K)$, then we can write
\begin{equation}
\appr_v(x,K)\>:=\>\{\appr_v(x,K)_{\gamma}\mid \gamma\in vK
\mbox{ and } \appr_v(x,K)_{\gamma}\ne\emptyset\}\>.
\end{equation}

\pars
\begin{lemma}                               \label{imme}
Let $(L|K,v)$ be an extension of valued fields.
\sn
1) If $x\in K$, then $\appr_v(x,K)$ is trivial, hence not immediate.
\sn
2) If $x\in L\setminus K$, then $\appr_v(x,K)$ is immediate if and only if 
the set $v(x-K)$ has no maximal element.
\sn
3) If $\appr_v(x,K)$ is immediate, then its support is equal to $v(x-K)$.
\sn
4) The extension $(L|K,v)$ is immediate if and only if $\appr_v(x,K)$ is
immediate for every $x\in L\setminus K$.
\end{lemma}
\begin{proof}
1): If $x\in K$, then 
\[
\bigcap_{\displaystyle\gamma\in \mbox{supp}\,\appr_v(x,K)} \appr_v(x,K)_\gamma\>=\>
\appr_v(x,K)_\infty\>=\>\{x\}\>\ne\>\emptyset\>.
\]
\n
2): Assume that $\appr_v(x,K)$ is immediate and that $c$ is an
arbitrary element of $K$. Then by definition there is some
$\gamma$ such that $c\notin \appr_v(x,K)_\gamma\ne\emptyset$, so
$v(x-c)<\gamma$. Choosing some $c'\in \appr_v(x,K)_\gamma\,$, we obtain
that $v(x-c)< \gamma\leq v(x-c')$. This shows that $v(x-K)$ has no maximal element.

Now assume that $v(x-K)$ has no maximal element. Take any $c\in K$; we have to 
show that there is some $\gamma\in \supp\,\appr_v(x,K)$ such that
$c\notin\appr_v(x,K)_\gamma\,$. By assumption, there are $c',c''\in K$ such that
\[
v(x-c'')\> >\>v(x-c')\> >\>v(x-c)\>.
\]
By the ultrametric triangle law we obtain that $v(x-c)< v(x-c')= v(c''-c')$. Hence 
$\appr_v(x,K)_{v(c''-c')}$ is nonempty, but does not contain $c$. 

\sn
3): This follows from Lemma~\ref{suppvx-K} together with part 2) of our lemma and 
part 1) of Lemma~\ref{va-K}.

\sn
4): This follows from part 2) of our lemma together with parts 2) and 3) of
Lemma~\ref{va-K}.
\end{proof}

The following result is a direct consequence of part 2) of the previous lemma 
together with part 5) of Lemma~\ref{va-K} and Lemma~\ref{ateq}.

\begin{corollary}                                \label{v>L-imm}
Take an extension $(L|K,v)$ and
$x,x'\in L$. If \ $\appr_v(x,K)$ is immediate, then
\begin{equation}                            \label{v>L-immeq}
\appr_v(x,K)=\appr_v(x',K) \>\Longleftrightarrow\> v(x-x')\geq\supp\,\appr_v(x,K)\>.
\end{equation}
\qed
\end{corollary} 

For our work with immediate approximation types, we introduce the
following useful notation. Take an immediate approximation type {\bf A}
over $(K,v)$, and some formula $\varphi$ in one free
variable. Then the sentence
\[
       \varphi(c) \mbox{ for } c\nearrow {\bf A}
\]
will denote the assertion
\[
\mbox{there is }\gamma\in vK \mbox{ such that }{\bf A}_{\gamma}\ne
\emptyset \mbox{ and } \varphi(c) \mbox{ holds for all }
c\in{\bf A}_{\gamma}\>.
\]
Note that if $\varphi_1(c)$ for $c\nearrow {\bf A}$ and $\varphi_2(c)$
for $c\nearrow {\bf A}$, then also $\varphi_1(c)\wedge \varphi_2(c)$ for
$c\nearrow {\bf A}$.
In the case of ${\bf A}=\appr_v(x,K)$, we will also write ``$c\nearrow
x$'' in place of ``$c\nearrow {\bf A}$''.

\pars
If $\gamma=\gamma(c)\in vK$ is a value that depends on $c\in K$ (e.g.,
the value $vf(c)$ for a polynomial $f\in K[X]$), then we will say
that \bfind{$\gamma$ increases for $c\nearrow x$} if there exists some
$\delta$ in the support of $\appr_v(x,K)$ such that for every
choice of $c'\in \appr_v(x,K)_{\delta}$,
\[
\gamma(c)>\gamma(c') \mbox{\ \ for\ \ } c\nearrow x\>.
\]
Similarly, we will say
that \bfind{$\gamma$ is fixed for $c\nearrow x$} if there exists some
$\delta$ in the support of $\appr_v(x,K)$ such that $\gamma$ is constant on 
$\appr_v(x,K)_\delta\,$.

\begin{lemma}                               \label{x:realiat} 
Take an immediate approximation type {\bf A} over $(K,v)$, an extension 
$(L|K,v)$, and an element $x\in L$. Then the following assertions are equivalent:
\sn
a) $x$ realizes {\bf A};
\sn
b) there is a cofinal subset $S\subseteq\supp{\bf A}$ such that for every 
$\gamma\in S$, $v(x-c)\geq \gamma$ for some $c\in {\bf A}_\gamma$;
\sn
c) for every $\gamma\in\supp{\bf A}$, $v(x-c)\geq \gamma$ for some $c\in
{\bf A}_\gamma$;
\sn
d) $v(x-c)$ is not fixed for $c\nearrow{\bf A}$.
\end{lemma}
\begin{proof}
a)$\Rightarrow$d): Assume that $x$ realizes {\bf A}, that is, ${\bf A}=\appr_v(x,K)$.
Take any $\gamma\in\supp {\bf A}$; we will show that for every $c\in 
{\bf A}_\gamma=\appr_v(x,K)_\gamma$ there is some $d\in \appr_v(x,K)_\gamma$ such that 
$v(x-c)< v(x-d)$. Since $\appr_v(x,K)$ is immediate, we know from part 2) of
Lemma~\ref{imme} that $v(x-K)$ has no maximal element. Hence there is $d\in K$ such 
that $v(x-c)<v(x-d)$. Since $v(x-c)\geq\gamma$, we find that $v(x-d)\geq\gamma$
and therefore, $d\in\appr_v(x,K)_\gamma\,$.

\sn
d)$\Rightarrow$c): Assertion d) means that for all $\gamma\in\supp{\bf A}$ there are
$c,d\in {\bf A}_\gamma$ such that $v(x-d)>v(x-c)$. This implies that
$v(x-d)>\min\{v(x-c),v(c-d)\}$, whence $v(x-c)=v(c-d)\geq\gamma$. 

\sn
c)$\Rightarrow$b): Trivial.

\sn
b)$\Rightarrow$a): Take any $\gamma\in S$. Then ${\bf A}_\gamma$ is a 
closed ultrametric ball of radius $\gamma$ in $(K,v)$. By assertion b), there is 
$c\in {\bf A}_\gamma$ with $v(x-c)\geq \gamma$. We obtain that ${\bf A}_\gamma=
B_\gamma(c,K)$. By the ultrametric triangle inequality,
\begin{eqnarray*}
d\in {\bf A}_\gamma=B_\gamma(c,K) &\Longleftrightarrow& v(c-d)\geq\gamma
\>\Longleftrightarrow\> v(x-d)\geq\gamma \\
&\Longleftrightarrow& d\in\{c'\in K\mid v(x-c')\geq\gamma\}=\appr_v(x,K)_\gamma\>.
\end{eqnarray*}
This shows that ${\bf A}_\gamma=\appr_v(x,K)_\gamma$ for all $\gamma\in S$. Since
$S$ is cofinal in $\supp {\bf A}$, part 1) of Lemma~\ref{uniqnest} shows that this
also holds for all $\gamma\in\supp {\bf A}$.
It remains to show that it also holds for $\gamma\notin\supp{\bf A}$. Since {\bf A}
is an immediate approximation type, we know that the intersection of all 
${\bf A}_\gamma$ for $\gamma\in\supp {\bf A}$ is empty. By what we have shown 
already, this is equal to the intersection of all $\appr_v(x,K)_\gamma$ for $\gamma
\in\supp {\bf A}$. If $\gamma\notin\supp{\bf A}$, then $\gamma>\supp{\bf A}$ and 
$\appr_v(x,K)_\gamma$ must be a subset of the intersection, hence empty and therefore
again equal to ${\bf A}_\gamma\,$. It follows that $\supp\,\appr_v(x,K)=\supp {\bf A}$, 
which has no maximal element. Hence by what we have shown together with Lemma~\ref{suppnle},
$\appr_v(x,K)={\bf A}$.
\end{proof}

\mn
%
%
%
\subsection{Algebraic and transcendental immediate approximation types}  
\label{sectalgtriat}
\mbox{ }\sn
In this section we present the approximation type version of 
Kaplansky's Theorems 2 and 3 (\cite{[Ka]}), which show that each immediate
approximation type can be realized in a simple immediate extension.

We consider an immediate approximation type {\bf A} over $(K,v)$ and a polynomial 
$f\in K[X]$. We will say that {\bf A fixes the value of} $f$\index{approximation
type fixes value of a polynomial} if there is $\delta\in\supp{\bf A}$ such that
the value $vf(c)$ is constant for $c\in {\bf A}_\delta\,$. If ${\bf A}=\appr_v(x,K)$, 
then this means that $vf(c)$ is fixed for $c\nearrow x$. If {\bf A} 
fixes the value of every polynomial in $K[X]$, then it is said to be 
{\bf transcendental}.\index{transcendental approximation type} 
If there is some $f\in K[X]$ whose value is not fixed by {\bf A}, 
then {\bf A} is said to be {\bf algebraic}.\index{algebraic approximation type}

If there exists any polynomial $f\in K[X]$ whose value is not fixed by
{\bf A}, then there also exists a monic polynomial of the same 
degree having the same property (since this property is not lost by
multiplication with non-zero constants from $K$). If $f(X)$ is a monic
polynomial of minimal degree such that {\bf A} does not fix the value of $f$, then 
it will be called an \bfind{associated minimal polynomial} for {\bf A} and its 
degree $\bd:=\deg f$ will be called the \bfind{degree of {\bf A}}. We set $\bd:=
\infty$ if {\bf A} is transcendental. Note that an
associated minimal polynomial $f$ for {\bf A} is always
irreducible over $K$. Indeed, if $g,h\in K[X]$ are of degree less than $\deg f$, then 
{\bf A} fixes the value of $g$ and $h$ and thus also of $g\cdot h$. 

\parm
We will now study the behaviour of polynomials with respect to
immediate approximation types $\appr_v (x,K)$. 
We use the Taylor expansion
\begin{equation}                             \label{Taylorexp}
f(X) = \sum_{i=0}^{n} \partial_i f(c) (X-c)^i
\end{equation}
where $\partial_i f$ denotes the $i$-th \bfind{Hasse-Schmidt derivative}.

We need the following lemma
for ordered abelian groups, which is a reformulation of Kaplansky's Lemma~4 in
\cite{[Ka]}. For archimedean ordered groups, it was proved
by Ostrowski in \cite{Os}.
\begin{lemma}                               \label{OST}
Take elements $\alpha_1,\ldots,\alpha_m$ of an ordered abelian group
$\Gamma$ and a subset $\Upsilon\subset\Gamma$ without maximal element.
Let $t_1,\ldots,t_m$ be distinct integers. Then there exists an element
$\beta\in \Upsilon$ and a permutation $\sigma$ of the indices
$1,\ldots,m$ such that for all $\gamma\in \Upsilon$, $\gamma\geq\beta$,
\[
\alpha_{\sigma(1)} + t_{\sigma(1)}\gamma >
\alpha_{\sigma(2)} + t_{\sigma(2)}\gamma > \ldots >
\alpha_{\sigma(m)} + t_{\sigma(m)}\gamma\>.
\]
\end{lemma}

\pars
If the immediate approximation type {\bf A} is of degree {\bf d} and
$f\in K[X]$ is of degree at most {\bf d}, then {\bf A} fixes the value
of every Hasse--Schmidt derivative $\partial_i f$ of $f$ ($1\leq i\leq\deg f$), 
since every such derivative has degree less than {\bf d}. So we can define
$\beta_i$ to be the fixed value $v\partial_i f(c)$ for $c\nearrow
x$.\glossary{$\beta_i$} In certain cases, a derivative may be
identically 0. In this case, we have $\beta_i =\infty$. However, the
Taylor expansion of $f$ shows that not all derivatives vanish
identically, and the vanishing ones will not play a role in our
computations.

By use of Lemma~\ref{OST}, we can now prove:
\begin{lemma}                        \label{C4+}
Take an immediate approximation type ${\bf A}=\appr_v(x,K)$ of degree
{\bf d} over $(K,v)$ and $f\in K[X]$ a polynomial of degree at most
{\bf d}. Further, let $\beta_i$ denote the fixed value $v\partial_i f(c)$
for $c\nearrow x$. Then there is a positive integer $\bh\leq\deg f$ such
that
\begin{equation}                                   \label{xf-}
\beta_{\bf h} + {\bf h}\cdot v(x-c)<\beta_i+i\cdot v(x-c)
\end{equation}
whenever $i\not={\bf h}$, $1\leq i\leq\deg f$ and $c\nearrow x$. Hence,
\begin{equation}                                \label{bhmin+}
v(f(x)-f(c)) = \beta_{\bhsc} + {\bf h}\cdot v(x-c)\;\;\;\mbox{ for }
c\nearrow x\>.
\end{equation}
Consequently, if {\bf A} fixes the value of $f$, then
\[
v(f(x)-f(c))>vf(x)=vf(c)\;\;\;\mbox{ for }c\nearrow x\>,
\]
and if {\bf A} does not fix the value of $f$, then
\[
vf(x)>vf(c)= \beta_{\bhsc} + {\bf h}\cdot v(x-c)\;\;\;\mbox{ for }
c\nearrow x\>.
\]
\end{lemma}
\begin{proof}
Set $n=\deg f$. We consider the Taylor expansion
\begin{equation}                  \label{Tv}
f(x) - f(c) = \partial_1 f(c)(x-c)+\ldots+\partial_n f(c)(x-c)^n
\end{equation}
with $c\in K$. We have that $v\partial_i f(c)(x-c)^i=\beta_i+i\cdot v(x-c)$ for
$c\nearrow x$. So we apply the foregoing lemma with $\alpha_i=\beta_i$
and $t_i=i$, and with $\Upsilon=\supp {\bf A}$ (which
has no maximal element since {\bf A} is an immediate approximation
type). We find that there is an integer $\bh\leq\deg f$ such that
$\beta_{\bf h}+ {\bf h} v(x-c) <\beta_i+iv(x-c)$ for $c\nearrow x$ and
$i\not={\bf h}$. This is Equation (\ref{xf-}), which in turn
implies Equation (\ref{bhmin+}).

If {\bf A} fixes the value of $f$, then $vf(x) \not=vf(c)$ is impossible
for $c\nearrow x$ since otherwise, the left hand side of (\ref{bhmin+})
would be equal to $\min\{vf(x),vf(c)\}$ and thus fixed while the right
hand side of (\ref{bhmin+}) increases for $c\nearrow x$. This
proves that $vf(x)=vf(c)$ and thus also $v(f(x)-f(c))\geq vf(x)$ for
$c\nearrow x$. But since the left hand side increases, we find that
$v(f(x)-f(c))>vf(x)$ for $c\nearrow x$.

If {\bf A} does not fix the value of $f$, then $vf(x)\not=vf(c)$ and
thus $v(f(x)-f(c))=\min\{vf(x),vf(c)\}$ for $c\nearrow x$. Since
$v(f(x)-f(c))$ increases for $c\nearrow x$ and $vf(x)$ is fixed,
the minimum must be $vf(c)$.
\end{proof}

\parm
If $g\in K[X]$ has a degree smaller than the degree of {\bf A}, then by
Lemma~\ref{C4+}, the value of $g(x)$ in $(K(x),v)$ is given by
$vg(x)=vg(c)$ for $c\nearrow x$. Since $g(c)\in K$, that means that the
value of $g(x)$ is uniquely determined by {\bf A} and the restriction of
$v$ to $K$. If $g$ is a nonzero polynomial, then $g(c)\not=0$ for
$c\nearrow x$ (since there is a nonempty ${\bf A}_\gamma$ which does
not contain the finitely many zeros of $g$, as {\bf A} is immediate).
Consequently, $g(x)\not=0$, which shows that the elements
$1,x,\ldots,x^{{\bf d} -1}$ are $K$-linearly independent.

We even know that $v(g(x)-g(c))> vg(x)$ for $c\nearrow x$. This means
that $(K,v)\subset (K+Kx+\ldots+ Kx^{{\bf d}-1},v)$ is an immediate
extension of valued vector spaces. If ${\bf d}=[K(x):K]<\infty$, then
$K(x)=K[x]=K+Kx+\ldots+ Kx^{{\bf d}-1}$, so by Lemma~\ref{ievagus}, the 
extension $(K(x)|K,v)$ is immediate. If ${\bf d}=\infty$, then
$(K,v) \subset (K[x],v)$ is immediate. But then again it follows that
the extension $(K(x)|K,v)$ is immediate. Indeed, if
$v(g(x)-g(c)) >vg(x)$ and $v(h(x)-h(c))>vh(x)$, then $vg(x)=vg(c)$,
$vh(x)=vh(c)$ and
\begin{eqnarray*}
v\left(\frac{g(x)}{h(x)}-\frac{g(c)}{h(c)}\right) & = &
v\left[g(x)h(c)-g(c)h(x)\right]-vh(x)h(c)\\
 & = & v\left[g(x)h(c)-g(c)h(c)+g(c)h(c)-g(c)h(x)\right]-vh(x)h(c)\\
 & = & v\left[(g(x)-g(c))h(c)+g(c)(h(c)-h(x))\right]-vh(x)h(c)\\
 & > & vg(x)h(x) -vh(x)h(x) \>=\> v\,\frac{g(x)}{h(x)}\>.
\end{eqnarray*}
We have proved:

\begin{lemma}                                           \label{CK1}
Take an immediate approximation type ${\bf A}=\appr_v(x,K)$ of
degree {\bf d} over $(K,v)$. Then the valuation on the valued
$(K,v)$-vector subspace 
\[
(K+Kx+\ldots+ Kx^{{\bf d}-1},v)
\]
of $(K(x),v)$ is uniquely determined by {\bf A} because
\[
vg(x)=vg(c)\;\;\;\mbox{ for }c\nearrow x
\]
for every $g(x)\in K+Kx+\ldots+ Kx^{{\bf d} -1}$. The elements
$1,x,\ldots,x^{{\bf d} -1}$ are $K$-linearly independent. In
particular, $x$ is transcendental over $K$ if ${\bf d}=\infty$.

Moreover, the extension $(K,v)\subset (K+Kx+\ldots+Kx^{{\bf d}-1},v)$
of valued vector spaces is immediate. In particular, if $\mbox{\bf d}=
\infty$ or if $\mbox{\bf d} =[K(x):K]<\infty$, then $(K[x]|K,v)$ is
immediate and the same is consequently true for the 
extension $(K(x)|K,v)$.
\end{lemma}

\begin{theorem} {\rm\ \ \ (Theorem 2 of \cite{[Ka]}, approximation
type version)}\label{KT2at}\n
For every transcendental immediate approximation type {\bf A} over
$(K,v)$ there exists a simple immediate transcendental extension
$(K(x),v)$ such that $\appr_v(x,K)={\bf A}$.

If $(K(y),v)$ is another extension field of $(K,v)$ such that
$\appr_v(y,K) = {\bf A}$, then $y$ is also transcendental over
$K$ and the isomorphism between $K(x)$ and $K(y)$ over $K$ that sends
$x$ to $y$ is valuation preserving.
\end{theorem}
\begin{proof}
We take $K(x)|K$ to be a transcendental extension and define the
valuation on $K(x)$ as follows. In view of the rule $v(g/h)=vg-vh$, it
suffices to define $v$ on $K[x]$. Take a nonzero polynomial $g\in K[X]$. By assumption,
{\bf A} fixes the value of $g$, that is, there is $\beta\in vK$ such
that $vg(c)=\beta$ for $c\nearrow {\bf A}$. We set $vg(x)=\beta$. 
Our definition implies that $vg\not=\infty$ for every nonzero $g\in K[x]$.

Take $g,h\in K(X)$. Again by our definition, $vg(x)=vg(c)$ and $vh(x)=vh(c)$ for
$c\nearrow {\bf A}$. Thus, 
\[
vg(x)h(x) \>=\> v(gh(x)) \>=\> v(gh(c))\>=\> vg(c)h(c)\>=\> vg(c)+vh(c) \>=\> 
vg(x)+vh(x)
\]
and 
\begin{eqnarray*}
v(g(x)+h(x))&=& v((g+h)(x)) \>=\> v((g+h)(c)) \>=\> v(g(c)+h(c))\\
&\geq& \min\{vg(c),vh(c)\} \>=\>\min\{vg(x),vh(x)\}
\end{eqnarray*}
for $c\nearrow {\bf A}$. So indeed, our definition
yields a valuation $v$ on $K(x)$ which extends the valuation $v$ of $K$.
Under this valuation, we have that ${\bf A}=\appr_v(x,K)$; this is seen as
follows. In view of Lemma~\ref{x:realiat}, it suffices to prove that for
every $\gamma\in\supp{\bf A}$, we have that $v(x -c_\gamma) \geq\gamma$
for each $c_\gamma\in {\bf A}_\gamma$. But this follows directly from
our definition of $v(x-c_\gamma)$ because $c\in {\bf A}_\gamma$ for 
$c\nearrow {\bf A}$ and thus $v(x-c_\gamma)=v(c-c_\gamma)\geq\gamma$.

From Lemma~\ref{CK1}, we now infer that $(K(x)|K,v)$ is an
immediate extension. Take an element $y$ in some valued field
extension of $(K,v)$ such that ${\bf A}=\appr(y,K)$. By hypothesis, the degree 
of {\bf A} is $\infty$. From Lemma~\ref{CK1} we can thus deduce that $y$ is
transcendental over $K$. Hence,
sending $x$ to $y$ induces an isomorphism from $K(x)$ onto $K(y)$.
We have to show that this isomorphism is
valuation preserving. For this, we only have to show that $vg(x)=vg(y)$
for every $g\in K[X]$. From Lemma~\ref{CK1} we infer that $vg(x)=vg(c)=vg(y)$ holds
for $c\nearrow {\bf A}$; this proves the desired equality.  
\end{proof}

\begin{corollary}                           \label{corKT2}
If $v$ is extended from a valued field $(K,v)$ to a simple field extension $K(y)$ 
such that $\appr_v(y,K)$ is a transcendental immediate approximation type, then $y$
is transcendental over $K$, the extension is uniquely determined by $\appr_v(y,K)$, 
and $(K(y)|K,v)$ is immediate.
\end{corollary}
\begin{proof}
By the foregoing theorem, there is an immediate extension $(K(x)|K,v)$
such that $\appr_v(x,K)=\appr_v(y,K)$, with $x$ transcendental over $K$. By
the same theorem, there is a valuation preserving isomorphism of $K(x)$
and $K(y)$ over $K$. This proves our assertions.
\end{proof}

The next lemma will show that every algebraic immediate approximation type 
is realized by some element in an algebraic valued field extension.
\begin{lemma}                               \label{nfvfreal}
Take an algebraic immediate approximation type {\bf A} over $(K,v)$,
a polynomial $f\in K[X]$ whose value is not fixed by {\bf A}, and a root
$a$ of $f$. Then there is an extension of $v$ from $K$ to $K(a)$ such
that ${\bf A}=\appr(a,K)$.
\end{lemma}
\begin{proof}
We choose some extension $w$ of $v$ from $K$ to $K(a)$.
We write $f(X)=d\prod_{i=1}^{\deg f}(X-a_i)$ with $d\in K$ and $a_i\in
\tilde{K}$. If for all $i$, the values $w(c-a_i)$ would be fixed for
$c\nearrow {\bf A}$, then {\bf A} would fix the value of $f$, contrary
to our assumption. Hence there is a root $a_i$ of $f$ such that $w(a_i-c)$
is not fixed for $c\nearrow{\bf A}$. Take some automorphism $\sigma$ of
$\tilde{K}|K$ such that $\sigma a=a_i$ and set $v:=w\circ\sigma$. Then $v$
extends the valuation of $K$, and $v(a-c)=w\circ\sigma(a-c)=w(\sigma
a-c)=w(a_i-c)$ is not fixed for $c\nearrow{\bf A}$. By
Lemma~\ref{x:realiat}, ${\bf A}=\appr(a,K)$.
\end{proof}

The following is the analogue of Theorem~\ref{KT2at} for immediate
algebraic approximation types.
\begin{theorem} {\rm\ \ \ (Theorem 3 of \cite{[Ka]}, approximation type
version)}\label{KT3at}\n
For every algebraic immediate approximation type {\bf A} over $(K,v)$ of
degree {\bf d} with associated minimal polynomial $f(X)\in K[X]$
and a root $a$ of $\,f$, there exists an extension of $v$ from
$K$ to $K(a)$ such that $(K(a)|K,v)$ is an immediate extension
and $\appr(a,K)={\bf A}$.

If $(K(b),v)$ is another extension field of $(K,v)$ such that
$\appr(b,K)={\bf A}$, then any field isomorphism between $K(a)$
and $K(b)$ over $K$ sending $a$ to $b$ preserves the valuation.
(Note that such an isomorphism exists if and only if $b$ is also a
root of~$\,f$.)
\end{theorem}
\begin{proof}
We consider the valuation $v$ of $K(a)$ given by Lemma~\ref{nfvfreal}. Then
$\appr(a,K)={\bf A}$. The fact that $(K(a)|K,v)$ is immediate follows
from Lemma~\ref{CK1}.

The last assertion of our theorem is shown in the same way as the
corresponding assertion of Theorem~\ref{KT2at}: if $\appr(a,K)=\appr
(b,K)$ and $g\in K[X]$ with $\deg g < \mbox{\bf d}$ then, again by
Lemma~\ref{CK1}, $vg(a)=vg(c)= vg(b)$ for $c\nearrow a$. Hence an
isomorphism over $K$ sending $a$ to $b$ will preserve the valuation.
\end{proof}

\pars
\begin{proposition}                     \label{cutdown}
Assume that $(K'(x)|K',v)$ and $(K'|K,v)$ are extensions where $x\notin K'$ and 
$(K,v)$ lies dense in $(K',v)$. Then $\appr_v(x,K')$ and $\appr_v(x,K)$ have the same
support, and if $\appr_v(x,K')$ is immediate, then so is $\appr_v(x,K)$.
If in addition $\appr_v(x,K')$ is transcendental, then so is $\appr_v(x,K)$.
\end{proposition}
\begin{proof}
Take $\gamma\in\supp\,\appr_v(x,K')$ and pick $c'_\gamma\in \appr_v(x,K')_\gamma\,$ so 
that $v(x-c'_\gamma)\geq\gamma$. Since $\gamma\ne\infty$ and $(K,v)$ lies dense 
in $(K',v)$, there is some $c_\gamma\in K$ such 
that $v(c'_\gamma-c_\gamma)>\gamma$. It follows that $v(x-c_\gamma)\geq\min\{
v(x-c'_\gamma),v(c'_\gamma-c_\gamma)\}\geq\gamma$, whence $c_\gamma\in\appr_v(x,K)_\gamma$
and $\gamma\in\supp\,\appr_v(x,K)$. Thus $\supp\,\appr_v(x,K')\subseteq
\supp\,\appr_v(x,K)$. The reverse inclusion follows from the fact that 
$B_\gamma(x,K(x))\cap K\subseteq B_\gamma(x,K'(x))\cap K'$. If $\appr_v(x,K')$ is 
immediate, then from
\[
\bigcap_{\displaystyle\gamma\in \mbox{supp}\,\appr_v(x,K)} \appr_v(x,K)_\gamma
\>\subseteq\> \bigcap_{\displaystyle\gamma\in \mbox{supp}\,\appr_v(x,K')} \appr_v(x,K')_\gamma\>=\>\emptyset
\]
we see that $\appr_v(x,K)$ is immediate.

\pars 
Now assume that $\appr_v(x,K')$ is transcendental immediate, and take $f\in K[X]$. 
Then $f\in K'[X]$ and so
there is $\gamma\in\supp\,\appr_v(x,K')=\supp\,\appr_v(x,K)$ such that $vf(c)$ is 
constant for $c\in\appr_v(x,K')_\gamma\,$. Since $\appr_v(x,K)_\gamma\subseteq
\appr_v(x,K')_\gamma$, the same holds for $c\in\appr_v(x,K)_\gamma\,$. This proves 
that $\appr_v(x,K)$ is transcendental.
\end{proof}

\mn
%
%
%
\subsection{Immediate approximation types versus pseudo Cauchy sequences}  
\label{sectiatpCs}
\mbox{ }\sn
The following proposition reveals the connection between immediate approximation 
types and pseudo Cauchy sequences. Take a valued field $(K,v)$, an immediate
approximation type {\bf A} over $(K,v)$, and a pseudo Cauchy sequence 
$(c_\nu)_{\nu<\lambda}$ in $(K,v)$. 
We will say that {\bf A} and $(c_\nu)_{\nu<\lambda}$ are \bfind{associated} if in
every extension $(L|K,v)$, $x\in L$ realizes {\bf A} if and only if $x$ is a limit 
of $(c_\nu)_{\nu<\lambda}$. 
Since trivial approximation types are not immediate, a pseudo Cauchy sequence in
$(K,v)$ associated with an immediate approximation type over $(K,v)$ cannot have a 
limit in $K$.

\begin{proposition}                       \label{pcsvsat}
1) Take an immediate approximation type {\bf A} over $(K,v)$, a pseudo Cauchy 
sequence $(c_\nu)_{\nu<\lambda}$ in $(K,v)$ without a limit in $K$. Then {\bf A} and 
$(c_\nu)_{\nu<\lambda}$ are associated if and only if for all $\nu<\lambda$,
\begin{equation}                           \label{A=B}
{\bf A}_{\gamma_\nu}=B_{\gamma_\nu}(c_\nu,K)\>.
\end{equation}
If this is the case, then $\supp{\bf A}$ is equal to the least initial segment 
of $vK$ that contains all $\gamma_\nu\,$.
\sn
2) Every immediate approximation type over $(K,v)$ is associated with some 
pseudo Cauchy sequence in $(K,v)$, which consequently has no limit in $K$.
\sn
3) Every pseudo Cauchy sequence in $(K,v)$ without a limit in $K$ is associated 
with a unique immediate approximation type over $(K,v)$.
\end{proposition}
\begin{proof}
1): Assume first that (\ref{A=B}) holds.
Since $(c_\nu)_{\nu<\lambda}$ in $(K,v)$ has no limit in $K$, we know from 
Lemma~\ref{pCs->nest} that the intersection of the ultrametric balls 
$B_{\gamma_\nu}(c_\nu,K)$, $\nu<\lambda$, is empty; hence the same holds for the
intersection over the ${\bf A}_{\gamma_\nu}$. Consequently, ${\bf A}_\gamma=
\emptyset$ when $\gamma>\gamma_\nu$ for all $\nu<\lambda$. This shows that the 
set $\{\gamma_\nu\mid\nu<\lambda\}$ is cofinal in $\supp {\bf A}$, which implies
the last assertion of part 1). We also obtain that the assumption of 
Lemma~\ref{x:realiat} is satisfied. Hence an element $x$ in some extension 
of $(K,v)$ realizes {\bf A} if and only if for every 
$\nu<\lambda$ we have that $v(x-c)\geq \gamma_\nu$ for some $c\in 
{\bf A}_{\gamma_\nu}=B_{\gamma_\nu}(c_\nu,K)$, whence also $v(x-c_\nu)\geq 
\gamma_\nu\,$. By condition (\ref{limitpCs}), this holds if and only if $x$
is a limit of $(c_\nu)_{\nu<\lambda}$. We have proved that {\bf A} and 
$(c_\nu)_{\nu<\lambda}$ are associated.

Assume now that {\bf A} and $(c_\nu)_{\nu<\lambda}$ are associated. From
Theorems~\ref{KT2} and~\ref{KT3} we know that there is at least one limit $x$ of 
$(c_\nu)_{\nu<\lambda}$ in some extension $(L,v)$ of $(K,v)$. By condition
(\ref{limitpCs}), we have that $v(x-c_\nu)\geq\gamma_\nu$ for all $\nu<\lambda$.
As $x$ is supposed to also realize {\bf A}, we must have that ${\bf A}=
\appr_v(x,K)$. Hence ${\bf A}_\gamma=\appr_v (x,K)_\gamma =\{c\in K\mid v(x-c)\geq\gamma\}$ for each $\gamma\in\supp {\bf A}$. In particular, $c_\nu\in 
\appr_v (x,K)_{\gamma_\nu}$ for all $\nu<\lambda$. By the ultrametric triangle law, 
\[
c\in B_{\gamma_\nu}(c_\nu,K)\>\Leftrightarrow\> v(c-c_\nu)\geq \gamma_\nu
\>\Leftrightarrow\> v(x-c)\geq \gamma_\nu \>\Leftrightarrow\> 
c\in \appr_v (x,K)_{\gamma_\nu} = {\bf A}_{\gamma_\nu}\>.
\]
This shows that (\ref{A=B}) holds.

\sn
2): Take an immediate approximation type {\bf A} over $(K,v)$. By (possibly
transfinite) induction we construct a pseudo Cauchy sequence 
$(c_\nu)_{\nu<\lambda}\,$. Pick 
any $c_0\in K$. Assume that we have constructed the sequence up to some $c_\nu$ 
that lies in ${\bf A}_{\gamma_\nu}$ for some $\gamma_\nu\in\supp {\bf A}$. Since 
{\bf A} is immediate, there is some $\gamma_{\nu+1}\in \supp
{\bf A}$ such that $c_\nu\notin {\bf A}_{\gamma_{\nu+1}}\,$, and we pick any 
$c_{\nu+1}\in {\bf A}_{\gamma_{\nu+1}}\,$. Assume that $\kappa$ is a limit 
ordinal and we have chosen $c_\nu\in {\bf A}_{\gamma_\nu}\setminus
{\bf A}_{\gamma_{\nu+1}}$ for all $\nu<\kappa$. If the values $\gamma_\nu$, 
$\nu<\kappa$, are cofinal in $\supp {\bf A}$, we are done and set $\lambda=
\kappa$. Otherwise, there is some $\gamma_\kappa\in\supp {\bf A}$ larger than all
$\gamma_\nu$, and we pick some $c_\kappa\in {\bf A}_{\gamma_\kappa}$. This 
procedure must stop at a limit ordinal $\lambda$ as the cardinality of the set of
values $\gamma_\nu$ is bounded by the cardinality of $\supp {\bf A}$. When it stops 
it means that the $\gamma_\nu$ we have constructed are cofinal in $\supp {\bf A}$.

Assume that $\rho<\sigma<\tau<\lambda$. By construction, $c_\rho\notin 
{\bf A}_{\gamma_\sigma}$ (since ${\bf A}_{\gamma_\sigma}\subseteq 
{\bf A}_{\gamma_{\rho+1}}$),  $c_\sigma\in {\bf A}_{\gamma_\sigma}\,$, and 
$c_\tau\in {\bf A}_{\gamma_\tau}\subseteq {\bf A}_{\gamma_\sigma}\,$. We obtain 
that $v(c_\sigma-c_\rho)<\gamma_\sigma\leq v(c_\tau-c_\sigma)$.
This shows that $(c_\nu)_{\nu<\lambda}$ is a pseudo Cauchy sequence.

Since $c_\nu\in {\bf A}_{\gamma_\nu}$ by construction, it follows that 
${\bf A}_{\gamma_\nu}=B_{\gamma_\nu}(c_\nu,K)$.
Hence condition (\ref{A=B}) holds for all $\nu<\lambda$, so by part 1), {\bf A} and 
$(c_\nu)_{\nu<\lambda}$ are associated.

\sn
3): Take a pseudo Cauchy sequence $(c_\nu)_{\nu<\lambda}$ in $(K,v)$ without a 
limit in $K$. From Lemma~\ref{pCs->nest} we know that $\cN:=(B_{\gamma_\nu}
(c_\nu,K))_{\nu<\lambda}$ is a nest of balls in $K$ and that the intersection over 
all balls in this nest is empty. By part 3) of Lemma~\ref{uniqnest} there is a uniquely 
determined full nest $\cA$ containing $\cN$, which satisfies $\bigcap\cA=
\bigcap\cN$. Consequently, $\cA$ is an immediate approximation type, uniquely 
determined by the pseudo Cauchy sequence $(c_\nu)_{\nu<\lambda}\,$. Condition 
(\ref{A=B}) holds by definition of $\cN$.
\end{proof}

\pars
The following result deals with an approximation type analogue of Cauchy sequences. We will 
say that {\bf A} is a \bfind{completion type} over $(K,v)$ if it is an immediate 
approximation type over $(K,v)$ with $\supp{\bf A}=vK$.
\begin{proposition}                        \label{realimmat}
Assume that {\bf A} is a completion type over $(K,v)$.
Then there is a unique element $y$ in the completion of $(K,v)$ that realizes {\bf A}. 
If $(L|K,v)$ is an extension containing two distinct elements $y,z$ which both realize
{\bf A}, then $v(y-z)>vK$. 
\end{proposition}
\begin{proof}
By Proposition~\ref{pcsvsat} there is a pseudo Cauchy sequence that is associated with 
{\bf A} and for which the values $\gamma_\nu$ are cofinal in $\supp{\bf A}=vK$. This
sequence is hence a Cauchy sequence, and by the definition of the completion, has a 
limit in the completion. This limit realizes {\bf A}. 

The last assertion follows from Corollary~\ref{v>L-imm}, and since the value group of the completion is equal to $vK$, it in turn proves the uniqueness in the first assertion. 
\end{proof}

\mn
%
%
\subsection{Properties of arbitrary approximation types}            \label{sectatrev}
\mbox{ }\sn
In Section~\ref{sectiat} we have exhibited the relation between immediate approximation
types and extensions of valuations to simple field extensions. In this section, we 
will have a closer look at arbitrary approximation types {\bf A} over a 
fixed valued field $(K,v)$, in particular those that are not immediate. By definition, 
the latter means that $\bigcap {\bf A}$ is not empty. 

Throughout, we will assume that {\bf A} is non-trivial. We define: {\bf A} is called
\bfind{residue-immediate} if it satisfies
\begin{equation}                                       \label{atri}
{\bf A}_\gamma^\circ\ne\emptyset\;\mbox{ for every } \gamma\in\supp {\bf A}\>,
\end{equation}
and it is called \bfind{value-extending} if in addition $\bigcap {\bf A}\ne\emptyset$.  
Note that if $\gamma\in\supp {\bf A}$ such that ${\bf A}_\gamma^\circ=\emptyset$, then 
$\gamma$ is the maximal element of $\supp {\bf A}$ since if there were $\delta\in\supp
{\bf A}$ such that $\gamma<\delta$, then the nonempty set ${\bf A}_\delta$ would be
contained in ${\bf A}_\gamma^\circ$.

Observe that the empty approximation type is value-extending.

\pars
Further, {\bf A} is called \bfind{value-immediate} if it satisfies
\begin{equation}                                       \label{atvi}
\mbox{for every $c\in K$ there is $\gamma\in\supp {\bf A}$ such that } c\in
{\bf A}_\gamma\setminus {\bf A}_\gamma^\circ\>,
\end{equation}
and it is called \bfind{residue-extending} if in addition $\bigcap {\bf A}\ne
\emptyset$.

\begin{example}                            \label{exreve}
Choose any initial segment $S$ of $vK$. Then 
\[
\{B_\gamma (0,K)\,,\, B_\gamma^\circ (0,K) \mid \gamma\in S\}
\]
is a value-extending approximation type. For any $\gamma_0\in vK$,
\[
\{B_\gamma (0,K)\mid \gamma\leq\gamma_0\} \cup
\{B_\gamma^\circ (0,K) \mid \gamma<\gamma_0\}
\]
is a residue-extending approximation type. If we adjoin $B_{\gamma_0}^\circ (0,K)$, then 
we obtain a value-extending approximation type. We will see that the value-extending 
approximation type describes an extension where $vx\notin vK$ induces a cut in $vK$ with 
lower cut set $S$, which is $\{\gamma\in vK\mid \gamma\leq\gamma_0\}$ in the last 
example. In contrast, the residue-extending approximation type describes an extension 
where $vx=\gamma_0\in 
vK$ and for any $d\in K$ such that $vdx=0$ we have that $dxv\notin Kv$. We see that the
difference between the two cases is revealed by an open ball; this is the reason why we 
take approximation types to contain both closed and open balls.
\end{example}

\begin{lemma}                                     \label{niatl}
1) Condition (\ref{atri}) holds if and only if ${\bf A}=\emptyset$ or
\begin{equation}                       \label{atrieq}
\bigcap{\bf A}\>=\> \bigcap_{\displaystyle\gamma\in\supp {\bf A}} {\bf A}_\gamma^\circ\>.
\end{equation}
\sn
2) {\bf A} is value-extending if and only if ${\bf A}=\emptyset$ or
\begin{equation}                       \label{charve}
\bigcap_{\displaystyle\gamma\in\supp{\bf A}} {\bf A}_\gamma^\circ\>\ne\>\emptyset\>.
\end{equation}
\sn
3) If $c\notin {\bf A}_\delta^\circ$ for some $\delta\in\supp {\bf A}$, then 
there is $\gamma\in\supp {\bf A}$ such that $c\in {\bf A}_\gamma\setminus 
{\bf A}_\gamma^\circ$. 
\sn
4) {\bf A} is residue-extending if and only if there is $\delta\in\supp{\bf A}$ such 
that $\bigcap {\bf A}={\bf A}_\delta$ and ${\bf A}_\delta^\circ=\emptyset$. If this 
is the case, then $\delta$ is the maximal element of $\supp{\bf A}$.
\sn
5) Every non-trivial approximation type is immediate, value-extending or 
residue-extending. These three properties 
are mutually exclusive.
\end{lemma}
\begin{proof}
Our assertions are trivial if ${\bf A}=\emptyset$, so we may assume that ${\bf A}
\ne\emptyset$.
\sn
1): The inclusion $\supseteq$ in (\ref{atrieq}) always holds. If condition 
(\ref{atri}) holds, then ${\bf A}_\gamma^\circ\ne\emptyset$ and thus
$\bigcap {\bf A}\subseteq {\bf A}_\gamma^\circ$ for all
$\gamma\in\supp{\bf A}$, hence the inclusion $\subseteq$ also holds in (\ref{atrieq}).
If condition (\ref{atri}) does not hold, then there is $\delta\in\supp {\bf A}$ such 
that ${\bf A}_\delta^\circ=\emptyset$. As ${\bf A}_\delta^\circ$ contains 
${\bf A}_\epsilon$ for every $\epsilon>\delta$, we then have that ${\bf A}_\epsilon
=\emptyset$, whence $\delta$ is the maximal element of $\supp{\bf A}$ and
$\bigcap{\bf A}={\bf A}_\delta\ne\emptyset=
\bigcap_{\gamma\in{\rm supp \bf A}} {\bf A}_\gamma^\circ$.
\sn
2): If {\bf A} is value-extending, then by part 1) of our lemma, 
\[
\bigcap_{\gamma\in {\rm supp\bf A}} {\bf A}_\gamma^\circ\>=\>\bigcap {\bf A}
\>\ne\>\emptyset\>.
\]
On the other hand, if (\ref{charve}) holds, then also (\ref{atrieq}) holds. Then by 
part 1) of our lemma, (\ref{atri}) holds, showing that {\bf A} is value-extending.
\sn
3): Take $c\in K$ and assume that $c\notin {\bf A}_\delta^\circ$ for some $\delta\in
\supp {\bf A}$. If $c\in {\bf A}_\delta$, then our assertion holds for $\gamma=
\delta$. Now assume that $c\notin {\bf A}_\delta$, pick $d\in {\bf A}_\delta$ and set
$\gamma=v(d-c)$. As $c\notin {\bf A}_\delta\,$, we have that $\delta>\gamma$ and 
therefore, ${\bf A}_\gamma=B_\gamma(d,K)$ and ${\bf A}_\gamma^\circ=B_\gamma^\circ
(d,K)$. Since $v(d-c)=\gamma$, $c\in {\bf A}_\gamma\setminus {\bf A}_\gamma^\circ$.
This proves part 3). 
\sn
4): Assume that {\bf A} is residue-extending, pick any $b\in\bigcap {\bf A}$ and $\delta
\in\supp {\bf A}$ such that $b\in {\bf A}_\delta\setminus {\bf A}_\delta^\circ$. Then $b 
\notin {\bf A}_\varepsilon$ for any $\varepsilon\in vK$ with $\varepsilon>\delta$ 
since ${\bf A}_\varepsilon\subseteq {\bf A}_\delta^\circ$. As $b\in\bigcap {\bf A}$, this 
shows that ${\bf A}_\varepsilon=\emptyset$, hence $\delta$ is the maximal element of 
$\supp{\bf A}$. It follows from this together with (\ref{atvi}) that for 
every $c\in {\bf A}_\delta$, we must have that $c\notin {\bf A}_\delta^\circ\,$, which 
proves that ${\bf A}_\delta^\circ=\emptyset$. 

For the converse, assume that there is $\delta\in\supp{\bf A}$ such 
that $\bigcap {\bf A}={\bf A}_\delta$ and ${\bf A}_\delta^\circ=\emptyset$. Since
$\delta\in\supp{\bf A}$, we have that $\bigcap {\bf A}\ne\emptyset$. Pick any $c\in 
K$. Since $c\notin \emptyset={\bf A}_\delta^\circ$, by part 3) there is 
$\gamma\in\supp {\bf A}$ such that $c\in {\bf A}_\gamma\setminus {\bf A}_\gamma
^\circ$, showing that {\bf A} is residue-extending. This finishes the proof of 
part 4).

\sn
5): Take a non-trivial approximation type {\bf A} which is not immediate, i.e., 
$\bigcap {\bf A}\ne\emptyset$. If {\bf A} is also not value-extending, then
there is $\delta\in\supp{\bf A}$ such that ${\bf A}_\delta^\circ=\emptyset$. 
By part 4), this implies that {\bf A} is
residue-extending.

Now we prove the second assertion of part 5).
If {\bf A} is immediate, then $\bigcap{\bf A}=\emptyset$, so {\bf A} can neither be
value-extending nor residue-extending. If {\bf A} is residue-extending, then by 
part 4) there is $\delta\in\supp{\bf A}$ such that $\bigcap {\bf A}=
{\bf A}_\delta$ and ${\bf A}_\delta^\circ=\emptyset$, hence for this $\gamma=\delta$, 
condition (\ref{atri}) is violated. This shows that {\bf A} cannot be 
value-extending.
\end{proof}

\begin{remark}                             
If {\bf A} is residue-extending, then by part 4) of the previous lemma, $\supp {\bf A}$
admits a maximal element $\delta$, and ${\bf A}_\delta^\circ=\emptyset$. Thus 
${\bf A}_\delta$ is the smallest ball in {\bf A}, and {\bf A} is generated by 
${\bf A}_\delta$.

If {\bf A} is value-extending and if $\supp {\bf A}$ admits a maximal element $\delta$, 
then ${\bf A}_\delta^\circ\ne\emptyset$ is the smallest ball in {\bf A} and {\bf A} is
generated by ${\bf A}_\delta^\circ$. However, not every value-extending approximation
type {\bf A} is such that $\supp {\bf A}$ admits a maximal element, in which case it is 
not generated by a single ball.
\end{remark}

\begin{lemma}                               \label{x:realat} 
Take an approximation type {\bf A} over $(K,v)$, an extension 
$(L|K,v)$, and an element $x\in L$. 
\sn
1) Assume that {\bf A} is value-extending and pick some $b\in\bigcap{\bf A}$. Then $x$
realizes {\bf A} if and only if $v(x-b)>\gamma$ for all $\gamma\in\supp{\bf A}$ and 
$v(x-c)<\varepsilon$ for all $\varepsilon\in vK\setminus\supp{\bf A}$ and all $c\in K$.
\sn
2) Assume that {\bf A} is residue-extending and pick some $b\in\bigcap{\bf A}$.
Then $x$ realizes {\bf A} if and only if
$v(x-b)=\max\supp{\bf A}$ and $v(x-c)\leq v(x-b)$ for all $c\in K$.
\end{lemma}
\begin{proof}
1): The condition ``$v(x-c)<\varepsilon$ for all $\varepsilon\in vK\setminus\supp{\bf A}$ and 
all $c\in K$'' is equivalent to $\supp\,\appr_v(x,K)\subseteq\supp{\bf A}$. Hence it holds
when $\appr_v(x,K)={\bf A}$, and this equality also implies that for all $\gamma\in
\supp{\bf A}$ we have that $v(x-b)>\gamma$. 

For the converse, assume that $v(x-b)>\gamma$ for all $\gamma\in\supp{\bf A}$. Then
for all $\gamma\in\supp{\bf A}$, $b\in \appr_v(x,K)_\gamma^\circ$ and thus also 
$b\in \appr_v(x,K)_\gamma$, which implies that $\appr_v(x,K)_\gamma^\circ=
B_\gamma^\circ(b,K)={\bf A}_\gamma^\circ$ and $\appr_v(x,K)_\gamma=B_\gamma(b,K)=
{\bf A}_\gamma$, and moreover, $\supp{\bf A}\subseteq\supp\,\appr_v(x,K)$. By what we 
have shown in the beginning, assuming also the condition ``$v(x-c)<\varepsilon$ for all 
$\varepsilon\in vK\setminus\supp{\bf A}$ and $c\in K$'' yields the reverse inclusion and thus 
equality of the supports. From Lemma~\ref{ateq} we now infer that $\appr_v(x,K)={\bf A}$.

\pars
2): Since {\bf A} is residue-extending, we know from part 4) of Lemma~\ref{niatl} that
for $\delta=\max\supp{\bf A}$, $b\in {\bf A}_\delta$ and ${\bf A}_\delta^\circ=\emptyset$.
Hence if $\appr_v(x,K)={\bf A}$, then $v(x-b)=\delta=\max\supp{\bf A}$ and $v(x-c)\leq 
\delta=v(x-b)$ for all $c\in K$. For the converse, assume that the latter holds. Then 
$\appr_v(x,K)_\delta^\circ=\emptyset={\bf A}_\delta^\circ$, but $b\in\appr_v(x,K)_\gamma$
for all $\gamma\leq\delta$ and therefore, $b\in\appr_v(x,K)_\gamma^\circ$ for all 
$\gamma<\delta$. It follows that $\supp\,\appr_v(x,K)=\supp{\bf A}$, $\appr_v(x,K)_\gamma
=B_\gamma(b,K)={\bf A}_\gamma$ for all $\gamma\leq\delta$, and $\appr_v(x,K)_\gamma^\circ=
B_\gamma^\circ(b,K)={\bf A}_\gamma^\circ$ for all $\gamma<\delta$. From Lemma~\ref{ateq} 
we obtain that $\appr_v(x,K)={\bf A}$.
\end{proof}

\pars
The following results justify the names ``value-extending'' and ``residue-extending''.
\begin{lemma}                      \label{charvere}
Take any extension $(K(x)|K,v)$ and set ${\bf A}:=\appr_v(x,K)$. Then the following 
assertions hold:
\sn
1) The approximation type {\bf A} is value-extending if and only if there is 
$b\in K$ such that $v(x-b)\notin vK$. If this is the case, then
\begin{equation}                \label{suppve}
\supp{\bf A}\>=\> v(x-K)\cap vK \>=\> \{\gamma\in vK\mid \gamma<v(x-b)\}\>,
\end{equation}
and $\bigcap{\bf A}$ is the set of all elements $b\in K$ for which $v(x-b)\notin vK$,
or equivalently, $v(x-b)$
realizes the cut $(\supp{\bf A}\,,\,vK\setminus\supp{\bf A})$.

\sn
2) The approximation type {\bf A} is residue-extending if and only if there are 
$b,d\in K$ such that $vd(x-b)=0$ and $d(x-b)v\notin Kv$. If this is the case, then
\begin{equation}                \label{suppre}
\supp{\bf A}\>=\> v(x-K)\>=\> \{\gamma\in vK\mid \gamma\leq v(x-b)\}\>,
\end{equation}
and 
$\bigcap{\bf A}$
is the set of all $b\in K$ for which $v(x-b)=\max\supp{\bf A}$,
or equivalently, $d(x-b)v\notin Kv$ for any $d\in K$ with $vd=-\max\supp{\bf A}$.
\end{lemma}
\begin{proof}
1): Assume that {\bf A} is value-extending and pick $b\in \bigcap{\bf A}$. Suppose that 
$v(x-b)\in vK$. Then for $\gamma=v(x-b)$ we would have that $b\in {\bf A}_\gamma\setminus 
{\bf A}_\gamma^\circ$. However, by (\ref{atri}), ${\bf A}_\gamma^\circ
\ne\emptyset$, but $b\notin{\bf A}_\gamma^\circ$, which contradicts our choice of 
$b\in\bigcap{\bf A}$. 

To prove the converse, assume that there is $b\in K$ such that $v(x-b)\notin vK$. We
have that $b\in {\bf A}_\gamma$ and $b\in {\bf A}_\gamma^\circ$ for every 
$\gamma\in vK$ with $\gamma<v(x-b)$. Suppose that 
$\delta>v(x-b)$ and $d\in {\bf A}_\delta$. Then $v(x-d)\geq\delta>v(x-b)$,
whence $v(d-b)=v(x-b)\notin vK$, which is a contradiction, showing that 
${\bf A}_\delta=\emptyset$ and $\delta\notin\supp{\bf A}$. 
Consequently, $b\in\bigcap{\bf A}$ and 
(\ref{suppve}) holds. For $\gamma\in\supp{\bf A}$ we have that $\gamma<v(x-b)$,
hence $b\in {\bf A}_\gamma^\circ$ (and also $b\in {\bf A}_\gamma$ since
$b\in {\bf A}_\gamma^\circ\subseteq {\bf A}_\gamma$), 
so we have that ${\bf A}_\gamma^\circ\ne\emptyset$.
We have now proved that ${\bf A}$ is value-extending.

It remains to prove the last assertion of part 1). We have already shown that  
$b\in \bigcap {\bf A}$ if and only if $v(x-b)\notin vK$, and that the former implies 
(\ref{suppve}). This means that $\supp {\bf A}$ is the lower cut set of the cut induced by
$v(x-b)$ in $vK$, i.e., $v(x-b)$ realizes the cut $(\supp{\bf A}\,,\,vK\setminus\supp
{\bf A})$. Conversely, if this is true, then $v(x-b)\notin vK$. This finishes the proof of 
part 1).

\pars
2): Assume that {\bf A} is residue-extending. By part 4) of Lemma~\ref{niatl}
we may pick $\delta\in\supp {\bf A}$ such that $\bigcap {\bf A}=
{\bf A}_\delta$ and ${\bf A}_\delta^\circ=\emptyset$. Further, we pick $b\in
\bigcap {\bf A}$. Since $b\notin\emptyset={\bf A}_\delta^\circ$, we have that 
$v(x-b)=\delta\in vK$, so that we can pick $d\in K$ such that $vd=-\delta$. 
Then $vd(x-b)=0$ and from part 5) of Lemma~\ref{maxapp} we infer that $d(x-b)v\notin Kv$.

To prove the converse, assume that there are $b,d\in K$ such that $vd(x-b)=0$ and 
$d(x-b)v\notin Kv$. Then by part 5) of Lemma~\ref{maxapp}, $\delta:=v(x-b)=\max v(x-K)\in 
vK$. It follows that ${\bf A}_\delta^\circ=\emptyset$, $\bigcap {\bf A}=
{\bf A}_\delta\ne\emptyset$, and from part 4) of Lemma~\ref{niatl} we obtain that 
{\bf A} is residue-extending. It also follows that (\ref{suppre}) holds, which means that 
$v(x-b)=\max\supp{\bf A}$.

It remains to prove the last assertion of part 2). We have already shown that for any 
$b\in \bigcap {\bf A}$ we have that $v(x-b)=\max\supp{\bf A}$ and if $d\in K$ with 
$-vd=v(x-b)=\max\supp{\bf A}$, then $d(x-b)v\notin Kv$. We have also shown that the latter
implies that $v(x-b)=\max\supp{\bf A}$. Hence if $\delta=\max\supp{\bf A}$, then $b\in 
{\bf A}_\delta\,$. As {\bf A} is residue-extending, we know from part 4) of 
Lemma~\ref{niatl} that ${\bf A}_\delta^\circ=\emptyset$, which shows that $b\in\bigcap
{\bf A}$.
%
\end{proof}

For the sake of completeness, we include the following results.
\begin{proposition}                                     \label{niat1}
Take a non-trivial approximation type {\bf A} over a valued field $(K,v)$. Then the 
following assertions hold.
\sn
1) {\bf A} is immediate if and only if it is value-immediate and residue-immediate.
\sn
2) {\bf A} is value-immediate or residue-immediate.
\end{proposition}
\begin{proof}
1): Assume first that {\bf A} is immediate. Then $\supp {\bf A}$ has no maximal element.
Hence for every  $\gamma\in\supp {\bf A}$ there is $\delta\in\supp {\bf A}$ such that 
$\delta>\gamma$ and therefore $\emptyset\ne {\bf A}_\delta \subseteq 
{\bf A}_\gamma^\circ$. This proves that {\bf A} is residue-immediate. 

Take any $c\in K$. Since $\bigcap{\bf A}=\emptyset$ there is some $\gamma\in\supp{\bf A}$
such that $c\notin {\bf A}_\gamma^\circ$. Now it follows from part 3) of 
Lemma~\ref{niatl} that condition (\ref{atvi}) holds, so {\bf A} is value-immediate.

For the proof of the reverse implication, assume that {\bf A} is not immediate, so
$\bigcap {\bf A}\ne \emptyset$. Assume that {\bf A} is residue-immediate and thus 
value-extending. From part 5) of Lemma~\ref{niatl} we conclude that it cannot be
residue-extending. Since $\bigcap {\bf A}\ne \emptyset$ holds, condition (\ref{atvi})
must fail, showing that {\bf A} is not value-immediate.

\mn
2): This follows from part 1) of our lemma together with part 5) of Lemma~\ref{niatl}.
%
\end{proof}

\bn
%
%
%
\section{Realization of approximation types}
%

%
%
%
\subsection{Proof of Theorem~\ref{MTat1}}    \label{sectmth1}
\mbox{ }\sn
Let $(K,v_0)$ be an arbitrary valued field. 
Recall that we denote by $\cV$ the set of all extensions of $v_0$ 
to $K(x)$, and by $\cA$ the set of all non-trivial approximation
types over $(K,v_0)$. 
The following is a more precise version of Theorem~\ref{MTat1}.
\begin{theorem}                       \label{MTatenh1}
For every non-trivial approximation type {\bf A} there is an extension $v$ of $v_0$ to 
$K(x)$ such that $\appr_v (x,K)={\bf A}$. That is, the function 
\begin{equation}                                   \label{v->at}
\cV\>\longrightarrow \cA\>, \qquad  v\,\mapsto\, \appr_v (x,K)  
\end{equation}
is surjective. Moreover, the following assertions hold:
\sn
1) If {\bf A} is transcendental immediate, then the extension $v$ is uniquely determined and
immediate.

\sn
2) If {\bf A} is algebraic immediate with $\supp{\bf A}\ne v_0 K$, then the extension $v$ can 
be chosen to be value-transcendental or residue-transcendental. If {\bf A} is algebraic 
immediate with $\supp{\bf A}=v_0 K$, then the extension $v$ is uniquely determined and 
value-transcendental.

\sn
3) If {\bf A} is value-extending, then the extension can be constructed in the 
following way: Take $b\in\bigcap {\bf A}$, and take $\alpha$ in some ordered 
abelian group containing $v_0 K$ such that $\alpha$ is not a torsion element modulo $v_0 K$
and realizes the cut $\cC$ in $v_0 K$ that has lower cut set $\supp {\bf A}$. If
$v$ is the extension of $v_0$ to $K(x)$ obtained from
Corollary~\ref{pBcor} by assigning the value $\alpha$ to $z=x-b$, then
$\appr_v(x,K)={\bf A}$.

Every extension $(K(x)|K,v)$ constructed in this way is value-transcendental.

\sn
4) If {\bf A} is residue-extending, then the extension can be constructed 
in the following way: Take $b\in\bigcap {\bf A}={\bf A}_\delta$. If $v$ is the extension 
of $v_0$ to $K(x)$ obtained from Corollary~\ref{pBcor} by assigning the value $\delta$ to 
$z=x-b$, then $\appr_v(x,K)={\bf A}$.

Every extension $(K(x)|K,v)$ constructed in this way is residue-transcendental.
\end{theorem}
\begin{proof}
Let us assume first that {\bf A} is immediate. If it is transcendental, then by 
Theorem~\ref{KT2at} and Corollary~\ref{corKT2} there is a uniquely determined extension $v$
such that $\appr_v (x,K)={\bf A}$. This proves part 1).

Now assume that {\bf A} is algebraic. We first use 
Theorem~\ref{KT3at} to obtain an extension $v$ of $v_0$ to some algebraic extension $K(a)$
of $K$ such that $\appr_v (a,K)={\bf A}$. Then we choose some $\alpha>\supp {\bf A}$ that 
is an element either of $v_0 K$, or of some ordered abelian group containing $v_0 K$, in which 
case $\alpha$ should not be a torsion element modulo $v_0 K$. With $z:=x-a$ we use 
Corollary~\ref{pBcor}
to extend $v$ from $K(a)$ to $K(a,z)=K(a,x)$ by assigning the value $\alpha$ to the element 
$z$. Finally, we restrict this extension to $K(x)$ and infer from Corollary~\ref{v>L-imm} 
that $\appr_v (x,K)=\appr_v (a,K)={\bf A}$.

We can choose the extension $(K(a,x)|K(a),v)$ to be value-transcendental if we take
$\alpha\notin v_0 K$, and to be residue-transcendental if we take $\alpha\in v_0 K$; however,
if $\supp{\bf A}=v_0 K$, then the condition that $\alpha>\supp {\bf A}$ forces $\alpha\notin
v_0 K$ (in which case it is automatically non-torsion modulo $v_0 K$).
If $(K(a,x)|K(a),v)$ is value-transcendental, then so is 
$(K(a,x)|K,v)$, that is, $vK(a,x)/v_0 K$ is not a torsion group. Since $K(a,x)|K(x)$ is 
algebraic, $vK(a,x)/v_0 K(x)$ is a torsion group, so we conclude that $vK(x)/v_0 K$ cannot be
a torsion group, i.e., $(K(x)|K,v)$ must be value-transcendental. A similar argument 
shows that if $(K(a,x)|K(a),v)$ is residue-transcendental, then so is $(K(x)|K,v)$, 
using the fact that $K(a,x)v/K(x)v$ is algebraic. We have proved the first 
assertion of part 2).

Assume that {\bf A} is algebraic immediate with $\supp{\bf A}=v_0 K$. Then the extension we 
have already constructed is value-transcendental. It remains to show that it
is uniquely determined. To this end, let us assume that $(K(x)|K,v)$ is an extension
such that $\appr_v(x,K)={\bf A}$.
We extend $v$ from $K(x)$ to $\tilde{K}(x)$ and call this extension again $v$. Its
restriction to $\tilde K$ provides us with an extension of $v_0 $ from $K$ to $\tilde K$.

By Theorem~\ref{KT3at}, there are $b\in\tilde{K}$ and an extension $w$ of $v_0$ from 
$K$ to $K(b)$ such that $\appr_v(x,K)$ is realized by $b$ in $(K(b),w)$, that is,
$\appr_v(x,K)=\appr_w(b,K)$. We extend $w$ to $\tilde K$. As two extensions of $v_0$ 
from $K$ to $\tilde K$ are conjugate, there is $\sigma\in\Aut \tilde{K}|K$ such that
$w=v\circ\sigma$. We set $a:=\sigma b$. Then for every $c\in K$ we have that 
\[
v(a-c)\>=\> v(\sigma b-c)\>=\> v\sigma (b-c)\>=\> w(b-c)\>,
\]
which shows that $\appr_v(a,K)=\appr_w(b,K)=\appr_v(x,K)$. From 
Corollary~\ref{v>L-imm} we infer that 
\[
v(x-a)\>\geq\> \supp\,\appr_v(x,K)\>=\>v_0 K
\]
holds in $(K(a,x),v)$, i.e., $v(x-a)>v_0 K$. By Lemma~\ref{exteq}, this uniquely determines 
the extension of $v$ from $K(a)$ to $K(a,x)$.

If also $a'\in \tilde K$ realizes $\appr_v(x,K)=\appr_v(a,K)$, then again by 
Corollary~\ref{v>L-imm}, $v(x-a')\geq v_0 K$. It follows that $v(a-a')\geq\min\{v(x-a),
v(x-a')\}\geq v_0 K$. As $a-a'\in\tilde K$, this means that $v(a-a')=\infty$, i.e., $a=a'$. 
We have now shown that once we fix an extension $v$ of $v_0$ to $\tilde K$, then the 
element $a$ and the extension of $v$ from $K(a)$ to $K(a,x)$ are uniquely determined; then 
also by restriction, the extension of $v_0$ to $K(x)$ is uniquely determined.

If we choose another extension $v_1$ of $v_0$ to $\tilde K$, then with $w$ and $b$ as above,
we have that $w=v_1\circ\tau$ for some $\tau\in\Aut\tilde{K}|K$. Replacing $a$ by $a_1:=
\tau b$, the same construction as before yields a unique extension of $v_1$ from $K(a_1)$ to 
$K(a_1,x)$. The automorphism $\rho:=\tau\sigma^{-1}\in\Aut\tilde{K}|K$ is an isomorphism
from $K(a)$ to $K(a_1)$. As before for $v$, we have that $\appr_{v_1}(a_1,K)=\appr_w(b,K)=
\appr_v(x,K)$. The uniqueness statement in Theorem~\ref{KT3at} thus implies that $\rho$ is 
an isomorphism from $(K(a),v)$ to $(K(a_1),v_1)$. We extend $\rho$ to an isomorphism 
from $K(a,x)$ to $K(a_1,x)$ by setting $\rho x=x$. As the extensions of $v$ from $K(a)$ to
$K(a,x)$ and of $v_1$ from $K(a_1)$ to $K(a_1,x)$ are uniquely determined by the facts that
$v(x-a)>v_0 K$ and $v(x-a_1)>v_0 K$ and as $\rho(x-a)=x-a_1$, we find that $\rho$ is an 
isomorphism from $(K(a,x),v)$ to $(K(a_1,x),v_1)$. However, as the restriction of $\rho$ 
to $K(x)$ is the identity, the restrictions of $v$ and $v_1$ to $K(x)$ must coincide.
This finishes the proof of the uniqueness assertion in the second part of statement 2).

\pars
Now assume that {\bf A} is not immediate and pick some $b\in\bigcap{\bf A}$. By part 5) of
Lemma~\ref{niatl}, {\bf A} is value-extending or residue-extending. Let us assume first that 
it is value-extending. Consider the cut in $v_0 K$ whose lower cut set is $\supp{\bf A}$ and
choose an element $\alpha$ in some ordered abelian group containing $v_0 K$ that realizes 
this cut, i.e., $\alpha>\gamma$ for all $\gamma\in\supp{\bf A}$ and $\alpha<\varepsilon$ 
for all $\varepsilon\in v_0 K\setminus
\supp{\bf A}$. We can always choose $\alpha$ so that it is not a torsion element
modulo $v_0 K$; indeed, if the cut is induced by an element $\beta$ in the divisible hull of
$v_0 K$, then we can replace $\beta$ by $\beta+\iota$ where $\iota$ is an infinitesimal,
that is, $0<\iota<\gamma$ for all $\gamma$ in the divisible hull of $v_0 K$. With $z:=x-b$ 
we use Corollary~\ref{pBcor} to extend $v_0$ from $K$ to $K(z)=K(x)$ by assigning the value 
$\alpha$ to the element $z$. By our choice of $\alpha$, the so constructed extension 
$(K(x)|K,v)$ is value-transcendental. 

Moreover, for arbitrary $c\in K$ we have that 
$v(x-c)=v(x-b+b-c)=\min\{v(x-b),v(b-c)\}\leq v(x-b)<\varepsilon$ for all $\varepsilon\in
v_0 K\setminus\supp{\bf A}$. Hence by part 1) of Lemma~\ref{x:realat}, $\appr_v(x,K)=
{\bf A}$. This finishes the proof of part 3). 

\pars
Finally, assume that {\bf A} is residue-extending and set $\delta=\max\supp{\bf A}\in v_0 K$.
With $z:=x-b$ we use Corollary~\ref{pBcor} to extend $v_0$ from $K$ to $K(z)=K(x)$ by 
assigning the value $\delta$ to the element $z$. From Corollary~\ref{pBcor} it follows that 
the so constructed extension $(K(x)|K,v)$ is residue-transcendental.

By construction, for arbitrary $c\in K$ we have that $v(x-c)=v(x-b+b-c)=
\min\{v(x-b),v(b-c)\}\leq v(x-b)=\max\supp{\bf A}$. Hence by part 2) of 
Lemma~\ref{x:realat}, $\appr_v(x,K)={\bf A}$. This finishes the proof of part 4). 
\end{proof}

\begin{remark}
Assume that {\bf A} is algebraic immediate. Let us describe in more detail the restriction 
to $K(x)$ of the value-transcendental extensions $v$ on $K(a,x)$ that we have constructed 
in the above proof.

We fix an extension $v$ of $v_0$ to $\tilde K$.
The minimal polynomial $f$ of $b$ over $K$ is an associated minimal polynomial for 
$\appr_v(x,K)$. As $a$ and $b$ are conjugate over $K$, $f$ is also the minimal polynomial
of $a$ over $K$. We write 
\[
f(X)\>=\> \prod_{i=1}^{\deg f} (X-a_i) 
\]
in such a way that $v(x-a_i)\notin v\tilde{K}$ for $1\leq i\leq n\leq \deg f$ and $v(x-a_i)
\in v\tilde{K}$ for $n<i\leq \deg f$. Take any $i\leq n$. If $v(x-a)\ne v(x-a_i)$, then 
$v(a-a_i)=\min\{v(x-a), v(x-a_i)\}\notin v\tilde{K}$, a contradiction. Therefore,
\[
vf(x)\>=\> \sum_{i=1}^{\deg f} v(x-a_i) \>=\> nv(x-a) + \sum_{i=n+1}^{\deg f} 
v(x-a_i)\>=\> nv(x-a) + \alpha
\]
with $\alpha\in v\tilde{K}$. Hence $vf(x)$ is not a torsion element modulo $v_0 K$. Take any 
$h\in K[x]$ and present it in its $f$-adic expansion
\[
h\>=\> \sum_{i=1}^n g_i(x) f^i (x)\>,
\]
where every $g_i\in K[x]$ is of degree less than $\deg f=\deg \appr_v(x,K)$. By
Lemma~\ref{CK1}, the value of $g_i(x)$ is uniquely determined by $\appr_v(x,K)$ and lies in
$v_0 K$. Since $vf(x)$ is not a torsion element modulo $v_0 K$, it follows that 
\[
vh\>=\>\min_{1\leq i\leq n} vg_i(x) + i vf(x)
\]
and that $vK(x)=v_0 K\oplus\Z vf(x)$. Further, since the extension $(K(a)|K,v)$
is immediate, we have that $K(a)v=Kv_0\,$. From Corollary~\ref{pBcor} we know that $K(a,x)v=
K(a)v$, so we obtain that $K(a,x)v=Kv_0$ and thus also $K(x)v=Kv_0\,$.

\parm
Let us also mention that if the extension field $(K(a),v)$ we have constructed 
in the proof admits a transcendental immediate extension, then an immediate extension 
$(K(a,x)|K(a),v)$ can be constructed with $v(x-a)>\supp{\bf A}$. Then also 
$(K(a,x)|K,v)$ is immediate, and restricting $v$ to $K(x)$ yields an immediate extension 
$(K(x)|K,v)$ with $\appr_v(x,K)={\bf A}$.
\end{remark}

\begin{remark}
Value-extending approximation types determine the cut in $v_0 K$ that has to be filled by the
value of an element like $x-b$ if $x$ realizes the type, but if this cut can also 
be filled by an element that lies in 
the divisible hull of $v_0 K$, then the approximation type does not determine whether it has to 
be filled by such an element or an element that is non-torsion modulo $v_0 K$. We used this 
fact in the proof of Theorem~\ref{MTatenh1}. However, in the setting of this theorem this 
fact also implies that uniqueness of the extension $v$ will in general fail.

The situation is similar for residue-extending approximation types. In general, they cannot
determine whether the residue of an element like $d(x-c)$ has to be algebraic or 
transcendental over $Kv_0\,$. 

These problems do not appear in the setting of Theorems~\ref{MTat2} and~\ref{MTatenh2},
where $v_0 K$ is divisible and $Kv_0$ is algebraically closed.
\end{remark}

\mn
%
%
%
\subsection{Approximation types and model theoretic $1$-types}    \label{sect1t}
\mbox{ }\sn
In this section we exhibit the relation between approximation types and $1$-types and 
the information that can be inferred from model theoretic algebra.
For background on model theory and the notions we use we refer the reader to \cite{[CK],PD}.
In Theorem~\ref{mtreal} we will show that under certain additional assumptions, 
a given approximation type over $(K,v)$ can be realized by a transcendental element in some
elementary extension of $(K,v)$. (For simplicity, we will write ``$v$'' even for the 
valuation on $K$.)

Take a language $\cL$ and an $\cL$-structure $\cS$. Further, take a set $\cT$ of 
$\cL$-formulas in one variable $X$ with parameters from the universe of $\cS$. Let
$\cS'$ be an $\cL$-structure with substructure $\cS$ and $x$ an element of the 
universe of $\cS'$. We say that $x$ \bfind{realizes $\cT$ in $\cS'$} if 
$\varphi(x)$ holds in $\cS'$ for every $\varphi(X)\in\cT$.

A set $\cT$ as above is called a \bfind{type} (or {\bf $1$-type}) {\bf over $\cS$} if it 
is consistent. A criterion for this is that every finite subset of $\cT$ is realized in 
$\cS$ by some $a$ in the universe of $\cS$. The \bfind{type of $x$ over $\cS$}
is the set of all $\cL$-formulas $\varphi(X)$ in one variable $X$ with 
parameters from the universe of $\cS$ for which the sentence $\varphi(x)$ holds in
$\cS'$. 

When studying valued fields, we work with a language of valued fields $\cL$ that 
consists of the language $\cL_R=\{+,-,\cdot\,,0,1\}$ of rings or alternatively, 
the language $\cL_F=\{+,-,\cdot\,,0,1,\mbox{}^{-1}\}$ of fields, together with either 
a unary relation symbol $\cO(X)$ for membership in the valuation ring or a binary
relation symbol $X\mid_vY$ for valuation divisibility ($x\mid_vy\Leftrightarrow vx\leq 
vy$). Then ``$v(X-c)\geq vd$'', ``$v(X-c)> vd$'' and ``$v(X-c)= vd$'' are 
$\cL$-formulas with parameters $c,d$. If $(L|K,v)$ is any valued field extension and
$x\in L$, then the assertions ``$c\in \appr_v(x,K)_\gamma$'' and ``$c\in
\appr_v(x,K)_\gamma^\circ$'' are expressed by the sentences ``$v(x-c)\geq vd_\gamma$'' and 
``$v(x-c)> vd_\gamma$'', where $d_\gamma\in K$ with $vd_\gamma=\gamma$.

An ordered abelian group $\Gamma$ is \bfind{dense} if for all $\gamma,\delta\in\Gamma$
with $\gamma<\delta$ there is $\beta\in\Gamma$ such that $\gamma<\beta<\delta$.
\begin{proposition}
Take a valued field $(K,v)$, and for every $\gamma\in vK$ choose some $d_\gamma\in K$ 
with $vd_\gamma=\gamma$. Take a non-trivial approximation type {\bf A} over $(K,v)$, and for 
every $\gamma\in\supp{\bf A}$, pick some $c_\gamma\in{\bf A}_\gamma\,$. If {\bf A} is
value-extending, then assume that $vK$ is dense, and if {\bf A} is
residue-extending, then assume that $Kv$ is infinite.
\sn
If {\bf A} is immediate, then set
\[
\cT_{\bf A} \>:=\> \{v(X-c_\gamma)\geq vd_\gamma \mid \gamma\in\supp {\bf A}\}\>.
\]
%
%
If {\bf A} is value-extending, then pick some $b\in\bigcap{\bf A}$ and set
\begin{eqnarray*}
\cT_{\bf A} &:=& \{v(X-b)> vd_\gamma \mid \gamma\in\supp {\bf A}\}\\
&& \cup\> \{\neg\, v(X-c)\geq vd_\varepsilon \mid \varepsilon\in vK\setminus\supp {\bf A}
\mbox{ and } c\in K\}\>.
\end{eqnarray*}
If {\bf A} is residue-extending, then pick some $b\in\bigcap{\bf A}$, set
$\gamma_{\max} := \max\supp{\bf A}$ and
\[
\cT_{\bf A} \>:=\> \{v(X-b)= vd_{\gamma_{\max}}\}\>\cup\> 
\{\neg\, v(X-c) >vd_{\gamma_{\max}}\mid c\in K\}\>.
\]
In all three cases, the following assertions hold:
\sn
1) The set $\cT_{\bf A}$ is finitely realizable in $(K,v)$.
\sn
2) If $x$ is an element in any valued field extension 
of $(K,v)$ that realizes $\cT_{\bf A}$, then $x$ realizes {\bf A}, that is, 
$\appr_v(x,K)={\bf A}$. 
\end{proposition}
\begin{proof}
1): Take a finite subset $\cT_0\subseteq\cT_{\bf A}\,$. 
Assume first that {\bf A} is immediate.
Then we can write $\cT_0=\{v(X-c_{\gamma_i})\geq vd_{\gamma_i} \mid 1\leq i\leq n\}$.
We set $\gamma:=\max_i {\gamma_i}\in\supp{\bf A}$. 
For arbitrary $a\in {\bf A}_\gamma\,$, this 
implies that $a\in {\bf A}_{\gamma_i}$ and hence $v(a-c_{\gamma_i})\geq vd_{\gamma_i}$
for $1\leq i\leq n$. 
We have shown that $a$ realizes $\cT_0$ in $(K,v)$. 

\pars
Now assume that {\bf A} is value-extending. We assume that $\supp {\bf A}\ne\emptyset$;
the easy proof in the case of $\supp {\bf A}\ne\emptyset$ is left to the reader.
We write $\cT_0=\{v(X-b)> vd_{\gamma_i} \mid 1\leq i\leq n\}\cup
\{\neg\, v(X-c_j)\geq vd_{\varepsilon_j} \mid 1\leq j\leq m\}$. 
Set $\gamma:=\max_i \gamma_i\,$
and $\varepsilon:=\min\{\varepsilon_j\,,\,v(b-c_j)\mid 1\leq j\leq m \mbox{ and } 
\gamma<v(b-c_j)\}$. Since 
$vK$ is assumed to be dense, there is $\beta\in vK$ such that $\gamma<\beta<\varepsilon$.
Choose $b'\in K$ with $vb'=\beta$ and set $a:=b+b'$. Then for $1\leq i\leq n$ we have
that $v(a-b)=vb'=\beta>\gamma_i\,$. For $1\leq j\leq m$ we have that $vb'=\beta \ne
v(b-c_j)$, whence $v(a-c_j)=v(b'+b-c_j)=\min\{\beta,v(b-c_j)\}\leq\beta<\varepsilon_j\,$.
We have shown that $a$ realizes $\cT_0$ in $(K,v)$. 

\pars
Finally, assume that {\bf A} is residue-extending. Set $d:=d_{\gamma_{\max}}$. As we 
may pass to a larger subset of $\cT_{\bf A}$ as long as it remains finite, we can 
write $\cT_0=\{v(X-b)= vd\}\cup
\{\neg\,v(X-c_j) >vd\mid 1\leq j\leq m\}$. Since $Kv$ is assumed to be 
infinite, we can choose some $c\in K$ such that $vc=0$ and $cv\ne -d^{-1}(b-c_j)v$ for
all $j$ such that $vd=v(b-c_j)$. This implies that for those $j$ we have that 
$v(c+d^{-1}(b-c_j))=0$. Consequently, $v(c+d^{-1}(b-c_j))=\min\{vc,vd^{-1}(b-c_j)\}
\leq 0$ for all $j\in\{1,\ldots,m\}$. 

We set $a:=b+cd$. Then $v(a-b)=vcd=vd$ and $v(a-c_j)
=v(cd+b-c_j)=vd+v(c+d^{-1}(b-c_j))\leq vd\,$. Again, we
have shown that $a$ realizes $\cT_0$ in $(K,v)$. 

\sn
2): Take $x$ in any valued field extension of $(K,v)$ that realizes $\cT_{\bf A}$.

Assume first that {\bf A} is immediate. Then $v(x-c_{\gamma})\geq vd_{\gamma}=
\gamma$ for all $\gamma\in\supp{\bf A}$. By Lemma~\ref{x:realiat} we conclude that
$x$ realizes {\bf A}.

Now assume that {\bf A} is value-extending. Then $v(x-b)> vd_{\gamma}=
\gamma$ for all $\gamma\in\supp{\bf A}$, and 
$v(x-c)< vd_\delta =\delta$ for all $\delta\notin\supp {\bf A}$ and $c\in K$. By part 
1) of Lemma~\ref{x:realat} we conclude that $x$ realizes {\bf A}.

Finally, assume that {\bf A} is residue-extending. Then $v(x-b)=\gamma_{\max}=
\max\supp{\bf A}$ and $v(x-c)\leq v(x-b)$ for all $c\in K$.
By part 2) of Lemma~\ref{x:realat} we conclude that $x$ realizes {\bf A}.
\end{proof}

From general model theory we infer that all types over an $\cL$-structure $\cS$ are
realized in every $\card(\cS)^+$-saturated elementary extension of $\cS$, and that
such extensions always exist. We take $\cL$ to be the language of valued rings or fields, 
the $\cL$-structure $\cS$ to be a valued field $(K,v)$, and {\bf A} to be a non-trivial
approximation type over $(K,v)$. Then {\bf A} cannot be realized by an element in $K$.
Hence if an element $x$ in some elementary extension realizes {\bf A}, then it will not lie 
in $K$ and will thus be, by a well-known basic fact of model theory, be transcendental 
over $K$. This proves:
\begin{theorem}                           \label{mtreal}
Take a valued field $(K,v)$ and a non-trivial approximation type {\bf A} over $(K,v)$. If 
{\bf A} is value-extending, then assume that $vK$ is dense, and if {\bf A} is
residue-extending, then assume that $Kv$ is infinite. Then {\bf A} is realized in
some elementary extension of $(K,v)$ by an element $x$ that is transcendental over 
$K$.
\end{theorem}

\bn
%
%
%
\section{Pure and almost pure extensions}
In this section we work with a fixed valued field $(K,v_0)$ and an element 
$x$ that is transcendental over $K$. 

%
%
%
\subsection{Pure extensions}    \label{sectpure}
\mbox{ }\sn
Take an arbitrary extension $(K(x)|K,v)$ and
$t\in K(x)$. If $vt$ is not a torsion element modulo $v_0 K$, then
$t$ will be called a \bfind{value-transcendental element}. If $vt=0$
and $tv$ is transcendental over $Kv_0$, then $t$ will be called a
\bfind{residue-transcendental element}. Further, $t$ will be called a
\bfind{valuation-transcendental element} if it is value-transcendental
or residue-trans\-cendental. 
In \cite{KTrans} we defined an extension $(K(x)|K,v)$
to be \bfind{pure} ({\bf in} $x$), if one of the following cases holds:
\sn
$\bullet$ \ for some $b,d\in K$, $d(x-b)$ is valuation-transcendental,
\sn
$\bullet$ \ $\appr_v(x,K)$ is a transcendental immediate approximation type.
\sn
Note that if $d(x-c)$ is value-transcendental, then we may in fact choose $d=1$.

If the extension $(K(x)|K,v)$ is pure, then we will also say that $v$ is a \bfind{pure 
extension} of $v_0$ from $K$ to $K(x)$.
\begin{lemma}                               \label{purevgrf}
Take any extension $(K(x)|K,v)$ and $b,d\in K$. 
\sn
1) Assume that $d(x-b)$ is value-transcendental. Then $\appr_v(x,K)$ is 
value-extending, and the valuation $v$ on $K(x)$ is uniquely determined by 
$(K,v_0)$ and the value $v(x-b)$. Further, we have that $vK(x)=v_0K\oplus\Z v(x-b)$ 
and that $K(x)v=Kv_0\,$.
\sn
2) Assume that $d(x-b)$ is residue-transcendental. Then $\appr_v(x,K)$ is 
residue-extending, and the valuation $v$ on $K(x)$ is uniquely
determined by $(K,v_0)$ and the fact that $d(x-b)v$ is transcendental over $Kv_0\,$.
Further, we have that $vK(x)=v_0K$ and that $K(x)v=Kv_0(d(x-b)v)$ is 
a rational function field over $Kv_0\,$. 
\sn
In both cases, $v_0K$ is pure in $vK(x)$ (i.e., $vK(x)/v_0 K$
is torsion free), and $Kv_0$ is relatively algebraically closed in $K(x)v$.
\end{lemma}
\begin{proof}
Lemma~\ref{charvere} shows that $\appr_v(x,K)$ is value-extending if $d(x-b)$ 
is value-transcendental, and that $\appr_v(x,K)$ is residue-extending if $d(x-b)$ 
is residue-transcendental.
All remaining assertions follow from Corollary~\ref{pBcor}.
\end{proof}

\parm
Here is the ``prototype'' of pure extensions:
\begin{proposition}                               \label{acpure}
If $K$ is algebraically closed and $x$ is transcendental over $K$, 
then every extension $(K(x)|K,v)$ is pure.
\end{proposition}
\begin{proof}
Assume first that the set $v(x-K)$ has no maximum. Then by part 2) of Lemma~\ref{imme},
$\appr_v(x,K)$ is immediate. Since $K$ is algebraically closed, Theorem~\ref{KT3at} 
shows that $\appr_v(x,K)$ must be transcendental. 

\pars
Now assume that the set $v(x-K)$ has a maximum, say, $v(x-b)$ with $b\in K$. 
Then by part 5) of Lemma~\ref{maxapp}, $v(x-b)\notin v_0 K$ or there is $d\in K$ 
such that $(d(x-b))v\notin Kv_0$. Since $K$  is algebraically closed, $v_0 K$ is divisible.  
Hence if $v(x-b)\notin v_0 K$, then it cannot be a torsion element modulo $v_0 K$, which implies 
that $(K(x)|K,v)$ is value-transcendental.

Since $K$  is algebraically closed, $Kv_0$ is also algebraically closed. Hence if the 
residue $(d(x-b))v$ is not in $Kv_0$, then it must be transcendental over $Kv_0$, which implies 
that $(K(x)|K,v)$ is residue-transcendental.

In all three cases, the extension is pure by definition.
\end{proof}

We will call {\bf A} a \bfind{pure approximation type} if it is a transcendental immediate,
value-extending or residue-extending approximation type. We denote by $\cV_p$ the set of 
all pure extensions of $v_0$ to $K(x)$ and by $\cA_p$ the set of all pure approximation 
types. For pure extensions, a stronger form of part 5) of Lemma~\ref{niatl} holds:
%
\begin{theorem}                       \label{pureextat}
For a pure extension $(K(x)|K,v)$, there are the following three mutually exlusive 
cases: immediate, value-transcendental, residue-transcendental, and the following assertions hold:
\sn
1) The extension is immediate if and only if $\appr_v(x,K)$ is immediate. In this 
case, the extension $v$ to $K(x)$ is uniquely determined by $\appr_v(x,K)$.
\sn
2) The extension is value-transcendental if and only if $x-b$ is value-transcendental 
for some $b\in K$, and this holds if and only if $\appr_v(x,K)$ is value-extending.
If in addition $v_0 K$ is divisible, then the pure extension $v$ to $K(x)$ is uniquely 
determined by $\appr_v(x,K)$.
\sn
3) The extension is residue-transcendental if and only if $d(x-b)$ is 
residue-transcen\-dental for some $b,d\in K$, and this holds if and only if 
$\appr_v(x,K)$ is residue-extending. In this 
case, the pure extension $v$ to $K(x)$ is uniquely determined by $\appr_v(x,K)$.
\sn
4) If $v_0 K$ is divisible, then the function 
\begin{equation}                                   \label{vp->atp}
\cV_p\>\longrightarrow \cA_p\>, \qquad  v\,\mapsto\, \appr_v (x,K)  
\end{equation}
is a bijection.
\end{theorem}
\begin{proof}
1): If $(K(x)|K,v)$ is immediate, then $\appr_v(x,K)$ is immediate by part 4) of
Lemma~\ref{imme}. For the converse, assume that $\appr_v(x,K)$ is immediate. If 
there is a valuation-transcendental element $d(x-b)$, then 
Lemma~\ref{charvere} shows that $\appr_v(x,K)$ is value-extending or 
residue-extending, hence not immediate. Thus by the definition of ``pure extension'', 
$\appr_v(x,K)$ must be transcendental immediate. By Corollary~\ref{corKT2}, this 
implies that $(K(x)|K,v)$ is immediate, and that the valuation $v$ on $K(x)$ is uniquely 
determined by $\appr_v(x,K)$.
\sn
2)\&3): Assume that $(K(x)|K,v)$ is valuation-transcendental. Then it is not immediate, 
so by Corollary~\ref{corKT2}, $\appr_v(x,K)$ is not immediate. Hence
by the definition of ``pure extension'', there are $b,d\in K$ such that $d(x-b)$ is
valuation-transcendental. Assume that $(K(x)|K,v)$ is value-transcendental. Then this 
element cannot be residue-transcendental because 
otherwise from part 2) of Lemma~\ref{purevgrf} it would follow that $(K(x)|K,v)$ is
not value-transcendental. Therefore, $d(x-b)$ is value-transcendental. From this it 
follows by part 1) of Lemma~\ref{purevgrf} that $\appr_v(x,K)$ is value-extending.

Assume that $(K(x)|K,v)$ is residue-transcendental. Similarly as before, one now concludes 
that $d(x-b)$ is residue-transcendental. From this it follows by part 2) of
Lemma~\ref{purevgrf} that $\appr_v(x,K)$ is residue-extending.

Assume now that $\appr_v(x,K)$ is value-extending. Then by part 5) of Lemma~\ref{niatl},
it cannot be immediate. Then by 
the definition of ``pure extension'', there are $b,d\in K$ such that $d(x-b)$ is
valuation-transcendental. If it were residue-transcendental, then by what we have 
already shown, $\appr_v(x,K)$ were residue-extending, a contradiction. Hence $d(x-b)$ is
value-transcendental, which implies that $(K(x)|K,v)$ is value-transcendental. 

In order to show the uniqueness statement, assume that $v_0 K$ is divisible.
Since ${\bf A}:=\appr_v(x,K)$ is value-extending, we know from part 1) of 
Lemma~\ref{charvere} that for an arbitrarily chosen $b\in\bigcap{\bf A}$, $v(x-b)$ realizes 
the cut $(\supp {\bf A}\,,\,v_0 K\setminus\supp{\bf A})$ in $v_0 K$. If also $v'$ is an
extension to $K(x)$ with $\appr_{v'}(x,K)=\appr_v(x,K)$, then again by part 1) of 
Lemma~\ref{charvere}, $v'(x-b)$ realizes the cut $(\supp {\bf A}\,,\,v_0 K\setminus
\supp{\bf A})$. It follows from part 2) of Lemma~\ref{exteq} that $v$ and $v'$ are 
equivalent over $v_0\,$, and as we identify equivalent valuations, $v=v'$.
This finishes the proof of part 2). 

Finally, assume that $\appr_v(x,K)$ is residue-extending. 
Interchanging ``residue-trans\-cendental'' and ``value-transcendental'' in the proof we 
just gave, we obtain that $(K(x)|K,v)$ is residue-transcendental. In order to show the
uniqueness statement, take another pure extension $v'$ such that ${\bf A}:=\appr_v(x,K)=
\appr_{v'}(x,K)$. From part 1) of Lemma~\ref{charvere} we infer that for an arbitrarily 
chosen $b\in\bigcap {\bf A}$ we have that $v(x-b)=\max\supp {\bf A}=v'(x-b)$ and that 
for $d\in K$ with $vd=-\max\supp {\bf A}$, $d(x-b)v$ and $d(x-b)v'$ do not lie in 
$Kv_0\,$. Since both extensions are pure, we know from Lemma~\ref{purevgrf} that $Kv_0$ 
is relatively algebraically closed in $K(x)v$ and in $K(x)v'$, which shows that both 
$d(x-b)v$ and $d(x-b)v'$ are transcendental over $Kv_0\,$. We set $y=d(x-b)$ and obtain 
from Proposition~\ref{prelBour} that $v=v'$ on $K(y)=K(x)$. This
finishes the proof of part 3). 


\pars
Part 5) of Lemma~\ref{niatl} shows that the properties ``immediate approximation type'', 
``value-extending approximation type'' and ``residue-extending approximation type'' are
mutually exclusive. Hence by what we have proved so far, for pure extensions the properties
``immediate'', ``value-transcendental'' and ``residue-transcen\-dental'' are mutually
exclusive.

\sn
4): If {\bf A} is a transcendental immediate approximation type, then by 
Theorem~\ref{KT2at} and Corollary~\ref{corKT2}
there is a unique extension $v$ of $v_0$ to $K(x)$ such that $\appr_v(x,K)={\bf A}$. 
It follows that $v\in\cV_p\,$.

Now assume that {\bf A} is a value- or residue-extending approximation type. By 
Theorem~\ref{MTatenh1} there is an extension $(K(x)|K,v)$ such that $\appr_v(x,K)={\bf A}$,
and we can choose it to be value-transcendental if {\bf A} is a value-extending, and to be
residue-transcendental if {\bf A} is a residue-extending. In both cases, $v\in\cV_p\,$. 
This proves the surjectivity of the function (\ref{vp->atp}). The injectivity follows from 
what we have proven in parts 1), 2) and 3).
\end{proof}

\begin{remark}
Without the assumption that $v_0 K$ is divisible, the uniqueness statement in the 
value-transcendental case can in general not be achieved. Assume that $\gamma\in
v_0 K$ is not divisible by some $n\in\N$. Then the cut induced by $\frac{\gamma}{n}$ in
$v_0 K$ may be equal to the cut $(\supp {\bf A}\,,\, v_0 K\setminus\supp {\bf A})$. We can 
fill this cut with an element $\alpha$ that is not torsion modulo $v_0 K$ by choosing a
positive infinitesimal $\iota$ and setting $\alpha=\frac{\gamma}{n}-\iota$ or 
$\alpha=\frac{\gamma}{n}+\iota$. Assigning $\alpha$ to $z$ as in Corollary~\ref{pBcor}
will lead to two distinct extensions: if $c\in K$ with $vc=-\gamma$, then in the first 
case, $vcz^n=-n\iota<0$, and in the second case, $vcz^n=n\iota>0$.
\end{remark}

\mn
%
%
%
\subsection{Almost pure extensions}    \label{sectalpure}
\mbox{ }\sn
In this section we generalize the notion ``pure extension'' in order to also capture the 
case where the base field $(K,v_0)$ lies dense in its algebraic closure, i.e., the
algebraic closure $\tilde K$ lies in the completion of $(K,v_0)$ (and consequently, 
the completion is itself 
algebraically closed). First, we prove that this property does not depend on the 
chosen extension of $v_0$ to $\tilde{K}$:
\begin{lemma}                                 \label{densext}
Take a valued field $(K,v_0)$ and extensions $v_1$ and $v_2$ of $v_0$ to $\tilde{K}$. Then 
$(K,v_0)$ lies dense in $(\tilde{K},v_1)$ if and only if it lies dense in $(\tilde{K},v_2)$.
\end{lemma}
\begin{proof}
Assume that $(K,v_0)$ lies dense in $(\tilde{K},v_1)$. Take $a\in \tilde{K}$ and $\alpha
\in v_2\tilde K$. As both $v_1\tilde K$ and $v_2\tilde K$ equal the divisible hull of 
$v_0 K$, we know that $\alpha\in v_1\tilde K$. Since all extensions of $v_0$ to $\tilde K$
are conjugate, there is an automorphism $\sigma$ of $\tilde{K}|K$ such that $v_2=
v_1\circ\sigma$. By assumption, there is $c\in K$ such that $\alpha<v_1(\sigma a -c)=
v_1\circ\sigma(a-c)=v_2(a-c)$. This proves that $(K,v_0)$ lies dense in $(\tilde{K},
v_2)$. By symmetry, also the converse holds, which proves our assertion.
\end{proof}

We define the extension $(K(x)|K,v)$ to be \bfind{almost pure}
({\bf in} $x$) if it is pure or $\appr_v(x,K)$ is an algebraic completion type.

\begin{proposition}                               \label{denspure}
If $(K,v_0)$ lies dense in its algebraic closure, then $v_0 K$ is divisible, $Kv_0$ is
algebraically closed, and every extension $(K(x)|K,v)$ is almost pure.
\end{proposition}
\begin{proof}
We extend $v$ from $K(x)$ to $\widetilde{K(x)}$. Then the restriction of $v$ to 
$\tilde K$ is an extension of $v$ from $K$ to $\tilde K$. Lemma~\ref{densext} shows
that $K$ lies dense in $({\tilde K},v)$. 

The completion of $(K,v_0)$ is an immediate extension, and it contains $({\tilde K},v)$. We
know that the value group of an algebraically closed valued field is divisible, and its 
residue field is algebraically closed. Hence the same holds for $v_0 K=v\tilde{K}$ and 
$Kv_0=\tilde{K}v$. 

From Proposition~\ref{acpure} we know that the extension $(\widetilde{K(x)}|\tilde{K},v)$ 
is pure. Assume first that $\appr_v (x,\tilde{K})$ is a transcendental immediate
approximation type. Then by Lemma~\ref{cutdown}, also $\appr_v (x,K)$ is a 
transcendental immediate approximation type and $(K(x)|K,v)$ is pure.
\pars
Now we consider the case where $d(x-b)$ is valuation-transcendental for some $b,d\in
\tilde K$. Assume that $vd(x-b)=0$ and $d(x-b)v$ is transcendental over $\tilde{K}v=
Kv_0\,$. Since $d\in \tilde K$, we know that $v(x-b)=-vd\in v\tilde{K}=v_0 K$.
We choose $d'\in K$ such that $v(d-d')>vd$ and $b'\in K$ such that $v(b-b')>v(x-b)$.
It follows that $vd=vd'$ and
\begin{eqnarray*}
v(d(x-b)-d'(x-b'))&\geq& \min\{v(d(x-b)-d'(x-b))\,,\,v(d'(x-b)-d'(x-b'))\}\\
&=&\min\{v(d-d')+v(x-b)\,,\,vd'+v(b-b')\}\\
&>& vd+v(x-b) \>=\>0\>.
\end{eqnarray*}
Therefore, $d'(x-b')v=d(x-b)v$ is transcendental over $Kv_0$. This shows that if  
$(\widetilde{K(x)}|\tilde{K},v)$ is residue-transcendental, then $(K(x)|K,v)$ is
residue-transcendental and pure.
\pars
Finally, assume that $vd(x-b)\notin v\tilde{K}=v_0 K$, in which case we can assume that 
$d=1$. We distinguish two cases:
\sn
Case 1: there is $\alpha\in v_0 K$ such that $\alpha>v(x-b)$. Then we choose $b'\in K$ 
such that $v(b-b')>\alpha$ and obtain that $v(x-b')=\min\{v(x-b),v(b-b')\}=v(x-b)
\notin v_0 K$. In this case, $(K(x)|K,v)$ is value-transcendental and pure.
\sn
Case 2: $v(x-b)>v_0 K$. Then by Corollary~\ref{v>L-imm}, $\appr_v(x,K)=\appr_v(b,K)$. This 
implies that $\supp\,\appr_v(x,K)=\supp\,\appr_v(b,K)=v_0 K$, so $(K(x)|K,v)$ is 
value-transcendental and almost pure. 
\end{proof}

\pars
We will call {\bf A} an \bfind{almost pure approximation type} if it is a pure 
approximation type or an algebraic completion type (note that every transcendental 
completion type is already a pure approximation type). We denote by $\cV_{ap}$ the set 
of all almost pure extensions of $v_0$ to $K(x)$ and by $\cA_{ap}$ the set of all almost 
pure approximation types. 
\begin{theorem}                       \label{alpureextat}
For an almost pure extension $(K(x)|K,v)$, there are the following three mutually exlusive 
cases: immediate, value-transcendental, residue-transcen\-dental, and parts 2), and 3)
of Theorem~\ref{pureextat} hold, as well as:
\sn
1') The extension $(K(x)|K,v)$ is immediate if and only if $\appr_v(x,K)$ is 
transcendental immediate.

If $\appr_v(x,K)$ is algebraic immediate, then the extension $(K(x)|K,v)$ is
value-transcendental.
\sn
4') If $v_0 K$ is divisible, then the function 
\begin{equation}                                   \label{vap->atap}
\cV_{ap}\>\longrightarrow \cA_{ap}\>, \qquad  v\,\mapsto\, \appr_v (x,K)  
\end{equation}
is a bijection.
\end{theorem}
\begin{proof}
From Theorem~\ref{pureextat} we know that for pure extensions there are the three mutually
exlusive cases immediate, value-transcendental, residue-transcendental. The only almost 
pure extensions that are not pure occur when the approximation type is an algebraic 
completion type. In this case we know from part 2) of Theorem~\ref{MTatenh1} that if 
$(K(x)|K,v)$ is an extension in which $x$ realizes the approximation type, then it must be 
value-transcendental. The latter also proves the second assertion of part 1'). 
Parts 2), and 3) of Theorem~\ref{pureextat} hold because in these cases the extensions 
are value-transcendental and residue-transcendental, respectively, hence pure. We see that 
by the definition of almost pure extensions, the only remaining case where $(K(x)|K,v)$ can 
be immediate occurs when $\appr_v(x,K)$ is transcendental immediate. Conversely, when the 
latter is the case, then by Corollary~\ref{corKT2}, $(K(x)|K,v)$ is immediate. This finishes
the proof of part 1').

\pars
Now we prove part 4'). In view of the bijection stated in part 4) of
Theorem~\ref{pureextat}, we only have to deal with the valuations in $\cV_{ap}\setminus
\cV_p$ and the approximation types in $\cA_{ap}\setminus\cA_p$. If ${\bf A}\in \cA_{ap}
\setminus\cA_p\,$, then it is an algebraic completion type, and by part 2) of
Theorem~\ref{MTatenh1}, it is realized by a uniquely
determined extension $v$; by definition, $v\in \cV_{ap}\setminus\cV_p\,$. Hence 
$\cV_{ap}\setminus\cV_p\ni v\mapsto\appr_v (x,K)\in\cA_{ap}\setminus\cA_p$
is a bijection, which finishes the proof of part 4').
\end{proof}

\mn
%
%
%
\subsection{Proof of Theorem~\ref{MTat2}}    \label{sectmth2}
\mbox{ }\sn
The following is a more precise version of Theorem~\ref{MTat2}.
\begin{theorem}                       \label{MTatenh2}
Assume that $K$ is algebraically closed, or that $(K,v_0)$ lies dense in its algebraic
closure. Then the function (\ref{v->at}) is a bijection. If $K$ is algebraically closed, 
then all assertions of Theorem~\ref{pureextat} hold. If $(K,v_0)$ lies dense in its 
algebraic closure, then all assertions of Theorem~\ref{alpureextat} hold. 
\end{theorem}
\begin{proof}
If $K$ is algebraically closed, then $v_0K$ is divisible and Proposition~\ref{acpure} 
shows that each extension 
$(K(x)|K,v)$ is pure, hence $\cV=\cV_p$ and all assertions of Theorem~\ref{pureextat} hold. 
Combining the surjectivity of the function (\ref{v->at}) proven in Theorem~\ref{MTatenh1}
with the bijectivity of the function in part 4) of Theorem~\ref{pureextat} shows that 
$\cA=\cA_p$ and that the function (\ref{v->at}) is a bijection.

If $(K,v)$ lies dense in its algebraic closure, then by Proposition~\ref{denspure}, 
$v_0K$ is divisible and each
extension $(K(x)|K,v)$ is almost pure, hence $\cV=\cV_{ap}$ and all assertions of
Theorem~\ref{alpureextat} hold. Combining the surjectivity of the function (\ref{v->at}) 
with the bijectivity of the function in part 4') of Theorem~\ref{alpureextat} shows that 
$\cA=\cA_{ap}$ and that also in this case, the function (\ref{v->at}) is a bijection. 
\end{proof}

\bn

\end{document}